\documentclass[12pt]{amsart}

\usepackage{graphicx} 
\usepackage{xcolor}
\usepackage{adjustbox}
\usepackage{amsmath,amssymb,amsxtra,tikz,stackengine,bm,mathrsfs,bbm,tikz-cd,comment,hyperref,combelow,amsopn,float,anysize}
\usepackage{subcaption}
\usepackage{enumitem}
\usepackage{ytableau}
\usepackage{quiver}
\usetikzlibrary{calc}%
\usetikzlibrary{shapes}%
\usetikzlibrary{patterns}%
\usetikzlibrary{positioning}%
\usetikzlibrary{arrows.meta}
\usetikzlibrary{knots}
\usetikzlibrary{decorations.markings}
\usetikzlibrary{hobby}

\newcommand{\Hilb}{\mathrm{Hilb}}
\newcommand{\C}{\mathbb{C}}
\newcommand{\Z}{\mathbb{Z}}
\newcommand{\R}{\mathbb{R}}
\renewcommand{\P}{\mathbb{P}}

\newcommand{\CT}{\mathcal{T}}
\newcommand{\CJ}{\mathcal{J}}
\newcommand{\CL}{\mathcal{L}}
\newcommand{\CO}{\mathcal{O}}
\newcommand{\CI}{\mathcal{I}}

\newcommand{\sgn}{\mathrm{sgn}}

\newcommand{\spann}{\mathrm{span}}
\newcommand{\supp}{\mathrm{supp}}
\newcommand{\Spec}{\mathrm{Spec}}
\newcommand{\Pic}{\mathrm{Pic}}

\newcommand{\Row}{\mathrm{Row}}
\newcommand{\Col}{\mathrm{Col}}

\newcommand{\Bl}{\mathrm{Bl}}

\numberwithin{equation}{section}

\newtheorem{theorem}{Theorem}[section]
\newtheorem{proposition}[theorem]{Proposition}
\newtheorem{corollary}[theorem]{Corollary}
\newtheorem{lemma}[theorem]{Lemma}

\theoremstyle{definition}
\newtheorem{remark}[theorem]{Remark}

\newtheorem{definition}[theorem]{Definition}
\newtheorem{example}[theorem]{Example}

\newcommand{\sq}{\square}
\newcommand{\bsq}{\blacksquare}

\title{Newton-Okounkov bodies for nested Hilbert schemes}
\author{Ian Cavey}
\address{Department of Mathematics, University of Illinois Urbana-Champaign, 1409 W Green St,
Urbana, IL 61801}
\email{cavey@illinois.edu}
\author{Eugene Gorsky}
\address{Department of Mathematics, University of California Davis, One Shields Avenue, Davis, CA 95616}
\email{egorskiy@ucdavis.edu}
\author{Alexei Oblomkov}
\address{Department of Mathematics and Statistics,
University of Massachusetts,
Lederle Graduate Research Tower,
710 North Pleasant Street,
Amherst, MA 01003}
\email{oblomkov@math.umass.edu}
\author{Joshua P. Turner}
\address{Department of Mathematics,
University of British Columbia, 
1984 Mathematics Road,
Vancouver, BC Canada V6T 1Z2}
\email{jpturner@math.ubc.ca}

\begin{document}

\begin{abstract}
We study sections of line bundles on the nested Hilbert scheme of points on the affine plane. We describe the spaces of sections in terms of certain ideals introduced by Haiman, and find explicit bases for them by analyzing the trailing terms in some monomial order. As a consequence, we compute the Newton-Okounkov bodies for nested Hilbert schemes.
\end{abstract}

\maketitle

\section{Introduction}

In this paper, we study line bundles on the nested Hilbert schemes of points on the affine plane, and their spaces of sections. 

Recall that the Hilbert scheme of $n$ points $\Hilb^n(\C^2)$ is the moduli space of ideals $\CI\subset \C[x,y]$ such that $\dim \C[x,y]/\CI=n$ (or, equivalently, $0$-dimensional subschemes $Z\subset \C^2$ of length $n$). It carries the natural tautological bundle $\CT=\C[x,y]/\CI$, the natural line bundle $\CO(1)=\wedge^n\CT$ and its powers $\CO(m)$. 

The space of global sections of $\CO(m)$ was studied by Haiman \cite{Haiman,Haimanqt} who proved, among other things, that for $m\ge 0$ it agrees with $\sgn(m)$-component of the ideal $J^{m}\subset \C[x_1,y_1,\ldots,x_n,y_n]$ where
\begin{equation}
\label{eq: def J intro}
J = \bigcap_{i<j} (x_i-x_j,y_i-y_j).
\end{equation}
Here $\sgn(m)=\sgn^{\otimes m}$ denotes the one-dimensional representation of $S_n$ which is trivial for $m$ even and sign for $m$ odd.
In \cite{C} the first named author described bases for these spaces of sections by characterizing the set of all trailing term exponents of these polynomials. The asymptotic version of this result as $m\to\infty$ is a Newton-Okounkov body computation for $\Hilb^n(\C^2)$, also described in \cite{C}.

The main object of this paper is the {\em nested Hilbert scheme} $\Hilb^{n,n+1}(\C^2)$ defined as the moduli space of pairs of ideals
$$
\mathcal{J}\subset \mathcal{I}\subset \C[x,y]
$$
such that $\dim \C[x,y]/\CI=n$ and $\dim \C[x,y]/\CJ=n+1.$ Equivalently, it is the space of pairs of subschemes $(Z,Z')\in \Hilb^n(\C^2)\times \Hilb^{n+1}(\C^2)$ such that $Z\subset Z'$. It is known \cite{Cheah,Tikh} that $\Hilb^{n,n+1}\C^2$ is smooth of dimension $2n+2$. We have natural projections
\begin{equation}
\label{eq: nested projections intro}
\begin{tikzcd}
 & \Hilb^{n,n+1}(\C^2) \arrow[swap]{dl}{\pi_n} \arrow{d}{\pi} \arrow{dr}{\pi_{n+1}}& \\
\Hilb^n(\C^2) & \C^2 & \Hilb^{n+1}(\C^2) 
\end{tikzcd}
\end{equation}
which send a pair $(\CJ\subset \CI)$ respectively to $\CI,\supp(\CI/\CJ)$ and $\CJ$.
Define line bundles $$\CO(m,k) = \pi_n^* \CO(m)\otimes \pi_{n+1}^* \CO(k)$$ on $\Hilb^{n,n+1}(\C^2)$. Here is our first main result.

\begin{theorem}\label{thm:polynomialsections}
    For $m,k\ge 0$ the global sections of $\CO(m,k)$ are identified with the $\sgn(m+k)$-component of $J^{m+k}\cap I^k\subseteq \C[x_1,y_1,\dots,x_n,y_n,x,y]$, where the ideal $J$ is given by \eqref{eq: def J intro} and 
    \[  I = \bigcap_{i=1}^n (x_i-x,y_i-y). \]
\end{theorem}

To prove this result, we first observe that 
\begin{equation}
\label{eq: nested centered}
\Hilb^{n,n+1}(\C^2)\simeq \C^2\times \Hilb^{n,n+1}_0(\C^2)
\end{equation}
where $\Hilb^{n,n+1}_0(\C^2)=\pi^{-1}(0)$. We describe the explicit isomorphism in Lemma \ref{lem: factorization}, which allows us to focus on line bundles on $\Hilb^{n,n+1}_0(\C^2)$.

Next, we consider  the Hilbert scheme of points of $\Bl_0\C^2$, the blowup of $\C^2$ at the origin. Let $E$ be the exceptional divisor in the blowup. Our second main result identifies the sections of line bundles on $\Hilb^{n,n+1}_0(\C^2)$ and on $\Hilb^n (\mathrm{Bl}_0 \C^2)$.

\begin{theorem}\label{thm:birationalequiv}
    For all $m,k\in \Z$, there is an isomorphism
    \[ H^0\left(\Hilb^{n,n+1}_0(\C^2),\CO(m,k)\right) = H^0\left(\Hilb^n (\mathrm{Bl}_0 \C^2),\CO(m+k)\otimes \CO(kE)_n\right)  \]
    where $\CO(E)_n$ denotes the line bundle obtained by pulling back the symmetrization of $\CO(E)$ from the symmetric power of $\mathrm{Bl}_0 \C^2$.
\end{theorem}

The proof of Theorem \ref{thm:birationalequiv} essentially follows from a geometric argument: we show that the complements of certain codimension 2 subsets in $\Hilb^{n,n+1}_0(\C^2)$ and on $\Hilb^n (\mathrm{Bl}_0 \C^2)$ are isomorphic, see Proposition \ref{prop: codim 2} and Corollary \ref{cor:birationalequiv}. 

The spaces of sections of line bundles on $\Hilb^n(\Bl_0\C^2)$ also admit a succinct description.

\begin{theorem}
\label{thm: sections blowup intro}
Let $A$ be the space of antisymmetric polynomials in $\C[x_1,\ldots,x_n,y_1,\ldots,y_n]$. Then
$$
H^0\left(\Hilb^n (\mathrm{Bl}_0 \C^2),\CO(m+k)\otimes \CO(kE)_n\right)
$$
is isomorphic to the subspace of polynomials $f$ in $A^{m+k}$ such that each monomial term $x_1^{a_1}y_1^{b_1} \cdots x_n^{a_n}y_n^{b_n}$ of $f$ satisfies $a_i+b_i \geq k$ for all $i$. 
\end{theorem}

We deduce Theorem \ref{thm:polynomialsections} from Theorems \ref{thm:birationalequiv} and \ref{thm: sections blowup intro} by means of a certain isomorphism $\phi$ inspired by \eqref{eq: nested centered}, see Section \ref{sec: global sections} for all details.

Next, we describe the explicit bases for these spaces of sections. We use lexicographic order on monomials such that $x<y<x_1<\cdots<x_n<y_1<\cdots<y_n$.  A basis element is characterized by its trailing term with respect to this order. 

\begin{theorem}
\label{thm: basis nested intro}
For any $n\geq 1$ and $m,k\geq 0$, a monomial $x^a y^b x_1^{a_1}y_1^{b_1}\cdots x_n^{a_n}y_n^{b_n}$ is the trailing term of some polynomial $g\in  H^0\left(\Hilb^{n,n+1}(\C^2),\CO(m,k)\right)$ if and only if:
\begin{enumerate}
    \item $a,b\geq 0$
    \item $0\leq a_1\leq a_2 \leq \cdots \leq a_n$,
    \item for any $j=1,\dots,n-1$ for which $a_j = a_{j+1}$, we have $b_{j+1}\geq b_j+m+k$, and
    \item for each $j=1,\dots,n,$ we have $b_j\geq \max\{k-a_j,0\} + \sum_{i=1}^{j-1}\max\{m+k-(a_j-a_i),0\}.$
    \end{enumerate}
\end{theorem}

As a consequence of our proof of Theorem \ref{thm: basis nested intro}, we obtain in Corollary \ref{cor:surjectivity} new surjectivity results for a certain ring of global sections of line bundles $\CO(m,k)$ on $\Hilb^{n,n+1}(\C^2).$

All of the above ideals are homogeneous, with $\Z^2$ grading given by $\deg(x_i)=\deg(x)=q$ and $\deg(y_i)=\deg(y)=t$. Geometrically, this grading corresponds to the action of $(\C^{\times})^2$ on $\C^2$ (resp. $\Bl_0(\C^2)$), and on the corresponding Hilbert schemes and spaces of sections.
We can use Theorem \ref{thm: basis nested intro} to compute their Hilbert series.

\begin{corollary}
The Hilbert series of $H^0\left(\Hilb^{n,n+1}(\C^2),\CO(m,k)\right)$ is given by 
$$
H_{m,k}(q,t)=\frac{1}{(1-q)(1-t)}\sum_{\mathcal{P}(m+k,k)}q^{a_1+\ldots+a_n}t^{b_1+\ldots+b_n}
$$
where $\mathcal{P}(m+k,k)$ is the subset of $\Z_{\ge 0}^{2n}$ defined by the inequalities (2)-(4) above.
\end{corollary}

Alternatively, for $k,m>0$ one can use localization formula \eqref{eq: localization} to compute this Hilbert series, see Corollary \ref{cor: localization} for more details. The equality between the two formulas leads to surprising combinatorial identities which we explore in Section \ref{sec: character formulas}.

Finally, we describe the Newton-Okounkov bodies on the nested Hilbert scheme with respect to the corresponding valuation. Let $\Delta(m+k,k)\subseteq \R^{2n}$ denote the polyhedron defined by conditions (2) and (4) in Theorem \ref{thm: basis nested intro}. 

\begin{theorem}
    The Newton-Okounkov body of $\CO(m,k)$ on $\Hilb^{n,n+1}(\C^2)$ with respect to the valuation described in Section \ref{sec: NO bodies} is $\R_{\geq 0}^2\times \Delta(m+k,k)$.
\end{theorem}

We expect that the combinatorics of these polytopes is related to the birational geometry of the nested Hilbert scheme from wall-crossing studied by Nakajima and Yoshioka \cite{NY}.

\section*{Acknowledgments}

The work of E. G. was partially supported by the NSF grant DMS-2302305. The work of A. O. was partially supported by the NSF grant DMS-2200798.
The work of J. T. was supported in part by the Pacific Institute for the Mathematical Sciences  (PIMS).






\section{Birational relation}
\label{sec:birational-relation}
In this section we relate the geometries of the nested  Hilbert scheme on \(\C^2\) and the Hilbert scheme of the blow up \(\mathrm{Bl}_0\C^2\). 

\subsection{Nested Hilbert schemes}

Let $S$ be a surface, we denote by $\Hilb^n(S)$ the Hilbert scheme of $n$ points on $S$ and by \(U\subset \Hilb^n(S)\times S\) the universal scheme. We have two projections.
$$
\begin{tikzcd}
 & U \arrow{dl}{p_n} \arrow{dr}{p_S}& \\
\Hilb^n(S) & & S
\end{tikzcd}
$$
and the scheme-theoretic fiber of $p_n$ over $\CI$ is $\Spec\ \CO/\CI$. There is a tautological rank $n$ vector bundle $\CT=p_{n*}(\CO_{U})$ and the tautological line bundle $\CO(1)=\det 
\CT$. 
More generally, given a divisor \(E\) on \(S\) there is
rank \(n\) vector bundle \(\CO(E)^{[n]}=p_{n*}(\CO_{U}\otimes p_S^*(\CO(E))\) 
and the line bundle \(\CO(E)_n=\det(\CO(E)^{[n]})\). 
 If $E\subset S$ is a smooth irreducible curve then
the latter has a geometric description as the subsheaf of functions vanishing on  the locus
\(E_n\subset \Hilb^n(S)\) where \(E_n\) is the locus of ideal sheaves \(\mathcal{I}\) such that
\(\mathrm{supp}(\CO/\mathcal{I})\cap E\ne \emptyset\).
Then \(E_n\) is Cartier divisor and  \(\CO(E)_n=\CO(E_n)\).

\begin{proposition}\cite{Fogarty2}
The Picard group of $\Hilb^n(S)$ is freely generated by $\CO(1)$ and the divisors $E_n$ for $E\in \Pic(S)$, so
$$
\Pic(\Hilb^n(S))\simeq \Z\langle \CO(1)\rangle\oplus \Pic(S).
$$
\end{proposition}
In particular, for $S=\Bl_0\C^2$ the Picard group of $\Hilb^n(\Bl_0\C^2)$ is freely generated by $\CO(1)$ and $E_n$ where $E\subset \Bl_0\C^2$ is the exceptional divisor.

Next, we consider the nested Hilbert schemes and the diagram \eqref{eq: nested projections intro}.
Consider the natural projection $\pi:\Hilb^{n,n+1}(\C^2) \to \C^2$ sending a pair of subschemes $(Z,Z')$ to the extra support point of $Z'$, and let $\Hilb^{n,n+1}_0(\C^2) = \pi^{-1}(0)$. The following is well known but we provide the proof for completeness.
\begin{lemma}
\label{lem: factorization}
We have
\[ \Hilb^{n,n+1}(\C^2) \simeq \C^2 \times \Hilb^{n,n+1}_0(\C^2). \]
\end{lemma}
\begin{proof}
Given $v=(v_1,v_2)$, we can define the translation 
$$
\tau_v:\C^2\to \C^2,\ \tau_v(x,y)=(x+v_1,y+v_2).
$$
This translation extends to $\Hilb^{n,n+1}(\C^2)$ where we also denote it by $\tau_v$. We can define the maps
$$
\Hilb^{n,n+1}(\C^2) \to \C^2 \times \Hilb^{n,n+1}_0(\C^2),\ (Z,Z')\mapsto (\pi(Z,Z'),\tau_{-\pi(Z,Z')}(Z,Z'))
$$
and
$$
\C^2 \times \Hilb^{n,n+1}_0(\C^2)\to \Hilb^{n,n+1}(\C^2),\ (v,(Z,Z'))\mapsto \tau_{v}(Z,Z').
$$
Clearly, these are inverse to each other and the result folows.
\end{proof}


We set \(\delta\subset \Hilb^{n,n+1}_0(\C^2)\) be the divisor consisting of pairs \((\mathcal{I},\mathcal{I}')\) such that
\(\mathrm{supp}(\CO/\mathcal{I})\) contains \(0\). Let \(\partial=\partial_n\subset \Hilb^n(\C^2)\), \(\partial_{n+1}\subset \Hilb^{n+1}(\C^2)\)
be the big diagonal divisors, i.e. the loci of nonreduced subschemes in each. It is well-known that on these Hilbert schemes we have $\CO(\partial_n) = \CO(-2)$ and $\CO(\partial_{n+1}) = \CO(-2)$. Then we have \(\delta=\pi_{n+1}^{-1}(\partial_{n+1})\setminus\pi_{n}^{-1}(\partial_n)\) and may define the line bundle
\[\mathcal{L}=\CO(-\delta/2)=\pi_n^*(\CO(-1))\otimes \pi_{n+1}^*(\CO(1)). \]


More generally, we may consider line bundles
$$\
\CO(m,k) = \pi_n^* \CO(m)\otimes \pi_{n+1}^*\CO(k)
$$
In particular, we have 
\begin{equation}
\label{eq: O(m) L notation}
\pi_n^*\CO(m)\otimes \CL^{k}=\CO(m-k,k).
\end{equation}
By \cite{Ryan}, $\Pic(\Hilb^{n,n+1}(\C^2))=\Z^2$ is spanned by $\CO(m,k)$.


\subsection{Birational relation}

We claim that $\Hilb^{n,n+1}_0(\C^2)$ and $\Hilb^n (\mathrm{Bl}_0 \C^2)$ are isomorphic outside of a codimension 2 locus.
In particular, we construct a graph of the birational  map \[\beta_n: \Hilb_0^{n,n+1}(\C^2)\dashrightarrow \Hilb^n(\Bl_0\C^2).\]
First, we observe that \(\Hilb_0^{1,2}(\C^2)\) and \(\Bl_0\C^2\) are naturally isomorphic and set \(\beta_1\) to be this isomorphism.
As above, let \(\pi_n:\Hilb^{n,n+1}_0(\C^2)\to \Hilb^n(\C^2)\) be the natural projection, and \(\partial \subset \Hilb^n(\C^2)\) is the big diagonal divisor. 

Then
we can define a regular map \(\beta_n^\circ: \Hilb_0^{n,n+1}(\C^2)\setminus \pi_n^{-1}(\partial)\to \Hilb^n(\Bl_0\C^2)\) by taking \(n\)-th symmetric power of
the map \(\beta_1\). More precisely, the complement of $\pi_n^{-1}(\partial)$ consists of unordered $n$-tuples of distinct points $Z=(P_1,\ldots,P_n)$ and another subscheme $Z'$ such that $Z\subset Z'$ and $\supp(Z')=\supp(Z)\cup 0$. If all of $P_i$ are distinct from $0$, then $Z'$ is uniquely determined by $Z$ and we just send $(Z,Z')$ to the corresponding $n$-tuple of points on $\Bl_0\C^2$. If $P_i=0$ for exactly one $i$ then  we can interpret $(Z,Z')$ as a point in $\Hilb^{n-1}(\C^2\setminus 0)\times \Hilb^{1,2}(\C^2)$ and use the map $\beta_1$ to define its image.  

Respectively, we define  \(\Gamma\subset \Hilb_0^{n,n+1}(\C^2)\times \Hilb^n(\Bl_0(\C^2))\) to be the closure of
graph of the map \(\beta_n^\circ\).

\begin{proposition}
  \label{prop: codim 2}
  Let \(f:\Gamma\to \Hilb_0^{n,n+1}(\C^2)\) and \(s:\Gamma\to \Hilb^n(\Bl_0\C^2)\) be the natural projections. Then
  the birational morphism \(\beta_n=s\circ f^{-1}\) is regular in codimension \(2\). This map yields an isomorphism
  of the Picard groups:
  \[ \beta_n(\partial)=\partial,\quad \beta_n(\delta)=E_n,\quad \beta_n(\pi_n^*\CO(1))=\CO(1).\]
\end{proposition}
\begin{proof}
Let $V\subset \Hilb^{n,n+1}_0$ be the open locus  of pairs of subschemes $(Z,Z')$ such that $Z$ has multiplicity at most $1$ at the origin.
Clearly, $\Hilb_0^{n,n+1}(\C^2)\setminus \pi_n^{-1}(\partial)\subset V$ and
by the above discussion the map $\beta_n^{\circ}$ extends to $V$.

Let us prove that the complement of $V$ has codimension $2$. Indeed, the locus in $\Hilb^{n,n+1}_0$ where $Z$ has multiplicity $k\ge 2$ at the origin is locally modeled on $\Hilb^{n-k}(\C^2\setminus 0)\times \Hilb^{k,k+1}(\C^2,0)$ and has codimension
$$
2n-2(n-k)-k=k.
$$
Here we used the fact that $\dim\Hilb^{n,n+1}_0(\C^2)=2n$ (this follows, say, from Lemma \ref{lem: factorization}) and $\dim \Hilb^{k,k+1}(\C^2,0)=k$ \cite{Cheah}.

The image $U=\beta_n(V)$ is the locus in $\Hilb^n(\Bl_0\C^2)$ consisting of subschemes $Z''$ such that the total multiplicity of $Z''$ on the exceptional divisor $E$ is at most $1$. We claim that the complement of $U$ also has codimension $2$. Indeed, consider the locus in  $\Hilb^n(\Bl_0\C^2)$ where $Z''$ has $s$ points with multiplicities $k_1,\ldots,k_s$ on $E$. It is locally modeled on 
$$
\Hilb^{n-k_1-\ldots-k_s}(\Bl_0\C^2\setminus E)\times \prod_{i=1}^{s}\Hilb^{k_i}(\C^2,0)\times E^s 
$$
and has codimension
$$
2n-2(n-k_1-\ldots-k_s)-\sum_{i=1}^{s}(k_i-1)-s=\sum_{i=1}^{s}k_i.
$$


   The  statement about Picard groups follows since \(\beta^{\circ}_n(\partial\cap V)\)    and \(\beta^{\circ}_n(\delta\cap V)\) are open sets inside irreducible
   divisors \(\partial\) and \(E_n\), and $\CO(1)=-\frac{1}{2}\partial$.
   \end{proof}



Since the varieties \(\Hilb_0^{n,n+1}(\C^2)\) and \(\Hilb^n(\Bl_0(\C^2))\) are smooth and hence normal we obtain an important statement

\begin{corollary}
\label{cor:birationalequiv}
    For all $m,k\in \Z$, there is an identification
    \[ H^0(\Hilb^{n,n+1}_0(\C^2),\CO(m)\otimes \CL^{\otimes k}) = H^0(\Hilb^n (\mathrm{Bl}_0 \C^2),\CO(m)\otimes \CO(kE)_n)  \]
    where $\CO(E)_n$ denotes the line bundle obtained by pulling back the symmetrization of $\CO(E)$ from the symmetric power of $\mathrm{Bl}_0 \C^2$.
\end{corollary}

The fact that \(\Hilb^{n,n+1}_0(\C^2)\) and \(\Hilb^n(\mathrm{Bl}_0\C^2 )\) are isomorphic outside of  codimension \(2\) almost immediately
implies
\begin{proposition}
  \(\Hilb^{n,n+1}_0(\C^2)\) is Frobenius split.
\end{proposition}
\begin{proof}
  \(\Bl_0(\C^2)\) is Frobenius split since it is a toric surface. Then \(\Hilb^n(\Bl_0\C^2)\) is Frobenius split by \cite[Theorem 7.5.2]{BrionKumar}. Moreover from the previous proof we see that there is an open subset \(V\subset \Hilb^{n,n+1}_0(\C^2)\) that is isomorphic to a dense
  open set in \(\Hilb^n(\Bl_0\C^2)\). Since \(\Hilb^{n,n+1}_0(\C^2)\) is smooth and the complement to \(V\) is of codimension \(2\) the
  statement of the proposition follows from \cite[Lemma 1.1.7]{BrionKumar}.
\end{proof}

\begin{corollary}
\label{cor:vanishing}
  For any \(m>0, k>0\) we have \(H^i(\Hilb^{n,n+1}_0(\C^2),\CO(m,k))=0\) for all \(i>0\).
\end{corollary}

\begin{proof}
    Consider the nested Hilbert scheme $\Hilb^{n,n+1}(\C^2)$ as an open subset of $\Hilb^{n,n+1}(\P^2)$. For fixed $m,k>0$, the line bundle $\CO(m,k)\otimes \pi_n^*\CO(\ell H)_n$ on $\Hilb^{n,n+1}(\P^2)$ is ample for $\ell$ sufficiently large, where $H\subseteq \P^2$ is the line at infinity \cite{Ryan}. The restriction of $\pi_n^*\CO(\ell H)_n$ to $\Hilb^{n,n+1}(\C^2)$ is trivial, so we have that $\CO(m,k)$ is ample on $\Hilb^{n,n+1}(\C^2)$ for all $m,k>0$. By Lemma \ref{lem: factorization}, $\CO(m,k)$ is ample on $\Hilb^{n,n+1}_0(\C^2)$ as well. The Frobenius splitting of $\Hilb^{n,n+1}_0(\C^2)$ therefore implies that the higher cohomology of $\CO(m,k)$ vanishes \cite{BrionKumar}, as desired.
\end{proof}


\subsection{Localization formulas}

One can use the Atiyah-Bott localization formula to compute the holomorphic Euler characteristic of $\CO(m,k)$. Recall that the natural action of $(\C^{\times})^2$ on $\C^2$ extends to $\Hilb^n(\C^2)$ and $\Hilb^{n,n+1}(\C^2)$.

The fixed points in $\Hilb^n(\C^2)$ correspond to monomial ideals $I_{\lambda}$ which are labeled by Young diagrams $\lambda$ of size $n$. The fixed points in $\Hilb^{n,n+1}(\C^2)$ correspond to nested pairs of monomial ideals, labeled by pairs of Young diagrams $\lambda\subset \mu$, $|\lambda|=n$ and $|\mu|=n+1$. We denote the box $\mu\setminus\lambda$ by $\bsq$.

The $(q,t)$-character of the cotangent space to $\Hilb^{n,n+1}(\C^2)$ at $(\lambda,\mu)$ was computed  in \cite{Can,Cheah}, which implies the following.

\begin{proposition}(\cite[Theorem 2.7]{Can})
We have
\begin{equation}
\label{eq: localization}
\chi(\Hilb^{n,n+1}(\C^2),\CO(m,k))=\sum_{\mu=\lambda\cup \bsq} \frac{\bsq^{k}\prod_{\sq\in \lambda}\sq^{k+m}}{(1-q)(1-t)P_1(\lambda,\mu)P_2(\lambda,\mu)P_3(\lambda,\mu)}
\end{equation}
where 
$$
P_1(\lambda,\mu)=\prod_{\sq\in \mu\setminus\left(\Row(\bsq)\cup\Col(\bsq)\right)}\left(1-q^{-a(\sq)}t^{1+\ell(\sq)}\right)\left(1-q^{1+a(\sq)}t^{-\ell(\sq)}\right),
$$
$$
P_2(\lambda,\mu)=\prod_{\sq\in  \Row(\bsq)}\left(1-q^{-a(\sq)}t^{1+\ell(\sq)}\right)\left(1-q^{a(\sq)}t^{-\ell(\sq)}\right)
$$
$$
P_3(\lambda,\mu)=\prod_{\sq\in  \Col(\bsq)}\left(1-q^{-a(\sq)}t^{\ell(\sq)}\right)\left(1-q^{1+a(\sq)}t^{-\ell(\sq)}\right).
$$
Here $a(\sq)$ and $\ell(\sq)$ denote the arm and the leg lengths of a box $\sq$ in the larger diagram $\mu$, and  in the numerator we identify the boxes $\sq,\bsq$ with their $(q,t)$-weights.
\end{proposition}

\begin{corollary}
\label{cor: localization}
For $k,m>0$ the bigraded character of $H^0(\Hilb^{n,n+1}(\C^2),\CO(m,k))$  is given by \eqref{eq: localization}.
\end{corollary}

\begin{proof}
For $k,m>0$   the higher homology vanish by Corollary \ref{cor:vanishing}, so the Euler characteristic agrees with the character of $H^0$.
\end{proof}

\begin{example}
\label{ex: n=2 localization}
For $n=2$ we have
\begin{multline*}
\chi(\Hilb^{n,n+1}(\C^2),\CO(m,k))=\frac{1}{(1-q)(1-t)}\times\\
\Biggl[\frac{q^{m+3k}}{(1-q^{-2}t)(1-q^2)(1-q^{-1}t)(1-q)}+\frac{t^{m+3k}}{(1-qt^{-2})(1-t^2)(1-qt^{-1})(1-t)}+\\
\frac{q^{k}t^{m+k}}{(1-q)(1-t)(1-q^{-1}t^2)(1-qt^{-1})}+\frac{q^{m+k}t^k}{(1-q)(1-t)(1-q^{-1}t)(1-q^2t^{-1})}\Biggr]
\end{multline*}
\end{example}

\section{Global Sections}
\label{sec: global sections}

We start with some definitions.  Let $\sgn(m)=\sgn^{\otimes m}$ denote the one-dimensional representation of $S_n$ which is trivial for $m$ even and sign for $m$ odd. Similarly, if $V$ is a representation of $S_n$ then the $\sgn(m)$ component of $V$ is $V^{S_n}$ if $m$ is even and $V^{\sgn}$ if $m$ is odd. Consider the diagonal action of $S_n$ on $\C[x_1,\ldots,x_n,y_1,\ldots,y_n]$ which permutes $x_i, y_i$ simultaneously. We will also consider the action of $S_n$ on $\C[x_1,\ldots,x_n,y_1,\ldots,y_n,x,y]$ which  permutes $x_i, y_i$ simultaneously and fixes the variables $x$ and $y$. We define the bigrading on these polynomial rings by $\deg(x_i)=\deg(x)=q$ and $\deg(y_i)=\deg(y)=t$, and note that the action of $S_n$ preserves the grading. We will simply refer to $\deg$ as to degree.

\begin{definition}
Let $A\subset \C[x_1,\ldots,x_n,y_1,\ldots,y_n]$ be the subspace of antisymmetric polynomials. We define $J$ as the ideal generated by $A$.
\end{definition}

Note that the space of antisymmetric polynomials in 
$\C[x_1,\ldots,x_n,y_1,\ldots,y_n,x,y]$ with respect to the above action is $A[x,y]$.

\begin{definition}
\label{def: delta S}
Let $S\subseteq \Z^2_{\ge 0}$ be an $n$-element subset. We will always assume without loss of generality that the elements of $S$ are labeled in increasing lexicographic order: $S=\{(a_1,b_1),\ldots,(a_n,b_n)\}$, and define $\Delta_{S}=\det(x_i^{a_j}y_i^{b_j})$. 
\end{definition}

Note that a different choice of ordering for elements of $S$ yields the same $\Delta_S$ up to sign. One can think of $\Delta_S$ as antisymmetrization of a monomial $x_1^{a_1}y_1^{b_1}\cdots x_n^{a_n}y_n^{b_n}$, which implies that $\Delta_S$ form a homogeneous basis of $A$. 

\begin{definition}
We define the integer powers of $A$ and $J$ as follows. If $m\le 0$ then
$
J^m=\C[x_1,\ldots,x_n,y_1,\ldots,y_n]
$
and $A^m$ is the $\sgn(m)$ component of $J^m$. If $m>0$ then $J^m$ (resp. $A^m$) is the span of products of $m$-tuples of elements of $J$ (resp. A).
\end{definition}

In particular, $A^0=\C[x_1,\ldots,x_n,y_1,\ldots,y_n]^{S_n}$. The following result of Haiman (see also \cite[Lemma 3.12]{C}) describes some important properties of $J^m$ and $A^m$.

\begin{theorem}\cite{Haiman,Haimanqt}\label{thm:secC2HS}
For all $m\in \Z$ one has
$$
H^0(\Hilb^n(\C^2),\CO(m))=A^m.
$$
Furthermore, we have  
$$
J^m=\bigcap_{i<j}(x_i-x_j,y_i-y_j)^m
$$
and the $\sgn(m)$ component of $J^m$ coincides with $A^m$. 
\end{theorem}

This description of global sections of line bundles on $\Hilb^n(\C^2)$ can be generalized to $\Hilb^n(X)$ for any smooth toric surface $X$. This was done in \cite{C} in the case $X$ is projective, but we state here the more general result and give a self-contained proof. We will then transport the description of sections of line bundles on $\Hilb^n(\Bl_0\C^2)$ the nested Hilbert scheme using Corollary \ref{cor:birationalequiv}. 

The generalization of Theorem \ref{thm:secC2HS} to toric surfaces $X$ with $X\subseteq \C^2$ is easily obtained by localization.

\begin{corollary}\label{cor:SectionsOverOpens}
    Let $\C^*\times \C\subseteq \C^2$ be the open set defined by $x\neq 0$ and $(\C^*)^2\subseteq \C^2$ the open set defined by $xy\neq 0$. We have open subsets $\Hilb^n((\C^*)^2)\subseteq \Hilb^n(\C^*\times \C)\subseteq \Hilb^n(\C^2)$, and the sections of $\mathcal{O}(m)$ over these open subsets are given by the localizations
    \[ H^0(\Hilb^n(\C^*\times \C),\mathcal{O}(m)) = (A^m)_{x_1\cdots x_n}\subseteq \C[x_1^{\pm1},\dots,x_n^{\pm1},y_1,\dots,y_n],\]
    and 
    \[ H^0(\Hilb^n((\C^*)^2),\mathcal{O}(m)) = (A^m)_{x_1\cdots x_n y_1\cdots y_n}\subseteq \C[x_1^{\pm1},\dots,x_n^{\pm1},y_1^{\pm1},\dots,y_n^{\pm1}].\]
    Moreover, each of these localizations is given by the $\sgn(m)$ components of the ideal $\bigcap_{i<j}(x_i-x_j,y_i-y_j)^m$ considered as an ideal in the corresponding Laurent polynomial ring.
\end{corollary}

We define $\widetilde{A^m}=(A^m)_{x_1\cdots x_n y_1\cdots y_n}$ to be the space of sections of $\mathcal{O}(m)$ on $\Hilb^n((\C^*)^2).$

Now we use this local description to establish the generalization of Theorem \ref{thm:secC2HS} to arbitrary smooth toric $X$. For background on toric varieties, see \cite{Fulton}.

\begin{theorem}
    Let $D$ be a torus-invariant divisor on a smooth toric surface $X$ with corresponding polygon $P_D\subseteq \R^2$. The global sections $H^0(\Hilb^n(X),\mathcal{O}(D)_n\otimes \mathcal{O}(m))$ are the elements $f\in\widetilde{A^m}$ such that $f$, considered as a Laurent polynomial in each pair $x_i$ and $y_i$, is supported on $P_D$.
\end{theorem}

\begin{proof}
    Let $T=(\C^*)^2\subseteq X$ denote the open torus. Since $\mathcal{O}(D)$ restricts trivially to $T$, we also have that the restriction of $\mathcal{O}(D)_n$ from $\Hilb^n(X)$ to $\Hilb^n(T)$ is trivial. By Corollary \ref{cor:SectionsOverOpens}, the sections of $\mathcal{O}(m)\otimes \mathcal{O}(D)_n$ over $\Hilb^n(T)$ are therefore identified with $\widetilde{A^m}$. 

    Now we will again use Corollary \ref{cor:SectionsOverOpens} to determine which of these extend to global sections. Write $D = \sum c_j D_j$ where the $D_j\subseteq X$ are torus-invariant curves. Each $D_j$ corresponds to a ray in the fan corresponding to $X$, and we let $(\alpha_j,\beta_j)$ denote the corresponding ray generator. We have $\C^*\times \C\simeq U_j\subseteq X$, where $U_j$ denotes the open subset obtained by removing all $D_k$ for $k\neq j$ from $X$. The restriction of $\mathcal{O}(D)$ to $U_j$ coincides with that of $\mathcal{O}(c_jD_j)$, and the curve $D_j\cap U_j\subseteq U_j$ is defined by $x^{\gamma_j}y^{\delta_j}=0$ where $(\gamma_j,\delta_j)$ is any integer point such that $\alpha_j \gamma_j +\beta_j\delta_j=1$. The sections of $\mathcal{O}$ over $U_j$ are the Laurent polynomials $f(x,y)$ such that every term $x^ay^b$ appearing in $f$ satisfies $\alpha_j a+\beta_j b \geq 0$. The sections of $\mathcal{O}(c_jD_j)$ over $U_j$ are obtained by multiplying the sections of $\mathcal{O}$ by $(x^{\gamma_j}y^{\delta_j})^{c_j}$, after which the support constraint becomes $\alpha_j a+\beta_j b \geq c_j$.

    After a change of coordinates, Corollary \ref{cor:SectionsOverOpens} says that the sections of $ \mathcal{O}(m)$ on $\Hilb^n(T)$ that extend to $\Hilb^n(U_j)$ are those elements $f\in \widetilde{A^m}$ such that in each pair of variables $x_i$, $y_i$, any term $x_i^a y_i^b$ that appears in $f$ satisfies $\alpha_j a+\beta_j b\geq 0$. Sections of $\mathcal{O}(c_iD_i)_n\otimes \mathcal{O}(m)$ over $\Hilb^n(U_j)$ are obtained by multiplying these by the local equation $(x_1^{\gamma_j}\cdots x_n^{\gamma_j}y_1^{\delta_j}\cdots y_n^{\delta_j})^{c_j}$, after which the support constraint becomes $\alpha_j a+\beta_j b \geq c_j$ for the exponents appearing on each set of variables.

    Consider the restriction of $\mathcal{O}(D)_n\otimes \mathcal{O}(m)$ to $\Hilb^n(U_j)$. This line bundle coincides with the restriction of $\mathcal{O}(c_jD_j)_n\otimes \mathcal{O}(m)$. Therefore, the sections of the restriction of $\mathcal{O}(D)_n\otimes \mathcal{O}(m)$ over $\Hilb^n(T)$ that extend to $\Hilb^n(U_j)$ are the elements $f\in \widetilde{A^m}$ that satisfy the support constraint corresponding to $\mathcal{O}(c_jD_j)$ in each pair of variables $x_i$, $y_i$.

    Finally, we observe that the open sets $\Hilb^n(U_j)$ cover $\Hilb^n(X)$ in codimension $1$. Indeed, the complement of the open set $\Hilb^n(T)\subseteq \Hilb^n(X)$ is the union of the divisors consisting of ideals $\mathcal{I}$ meeting one of the curves $D_j\subseteq X$, and each of these divisors meets $\Hilb^n(U_j)$. Thus the global sections of $\mathcal{O}(D)_n\otimes \mathcal{O}(m)$ are the same as those sections over $\bigcup \Hilb^n(U_j)$, which we have shown above is as claimed.
\end{proof}

In particular, we have the following.

\begin{proposition}
\label{prop: sections blowup}
    The global sections of $H^0\left(\Hilb^n (\mathrm{Bl}_0 \C^2),\CO(m)\otimes \CO(kE)_n\right)$ consist of polynomials $f$ in $A^{m}$ such that each monomial term $x_1^{a_1}y_1^{b_1} \cdots x_n^{a_n}y_n^{b_n}$ of $f$ satisfies $a_i+b_i \geq k$ for all $i$. 
\end{proposition}

We denote the set of polynomials in $A^{m}$ satisfying the  condition in Proposition \ref{prop: sections blowup} by $A^{m}_{\geq k}$.
To realize these as sections on the nested Hilbert scheme, we substitute 
$x_i-x$ for $x_i$ and $y_i-y$ for $y_i$,
which corresponds to the change of coordinates that shifts the distinguished point $(x,y)$ to the origin.

\begin{definition}
We define the ring homomorphism
$$\phi: \C[x_1, y_1, \dots, x_n, y_n, x, y] \rightarrow \C[x_1, y_1, \dots, x_n, y_n, x, y]$$
    $$x_i \mapsto x_i-x, \, y_i \mapsto y_i-y,$$
    $$x \mapsto x, \, y \mapsto y.$$
\end{definition}

\begin{remark}
We think of $\phi$ as an algebraic incarnation of the isomorphism in Lemma \ref{lem: factorization}.  Indeed, $(x_i,y_i)$ are support points of the subscheme $Z$ while $(x,y)=\pi(Z,Z')$ is the extra support point of $Z'$ for  $(Z,Z')\in \Hilb^{n,n+1}(\C^2)$.  
\end{remark}

\begin{lemma}
\label{lem: A[x,y] span}
a) Suppose $f\in A^m$ is a homogeneous polynomial of degree $q^{d_1}t^{d_2}$, then
\begin{equation}
\label{eq: shifted f}
\phi(f)=\\
f(x_1,y_1,\ldots,x_n,y_n)+\sum_{a\le d_1,b\le d_2,(a,b)\neq (d_1,d_2)} x^ay^b g_{a,b}(x_1,y_1,\ldots,x_n,y_n)
\end{equation}
where $g_{a,b}\in A^m$ is homogeneous of degree $q^{d_1-a}t^{d_2-b}$.

b) For all $m\in \Z$ the map $\phi$ maps $A^m[x,y]$ isomorphically to itself.
\end{lemma}

\begin{proof}
a)  Observe that  the map $\phi$  preserves the degree and commutes with the action of $S_n$. If $f\in A$ is a homogeneous antisymmetric polynomial of degree $q^{d_1}t^{d_2}$ then  $\phi(f)$ is antisymmetric and belongs to $A[x,y]$, so $g_{a,b}\in A$. Similarly, if $f$ is symmetric then all $g_{a,b}$ are symmetric, this proves the statement for $m\le 1$. 

For $m>1$, observe that $\phi(f_1\cdots f_m)=\phi(f_1)\cdots \phi(f_m)$. If $f_i\in A$ for each $i$ then we can expand $\phi(f_i)$ as a polynomial in $x,y$ using \eqref{eq: shifted f}, and all coefficients belong to $A$. By multiplying such expansions, we get an expansion of $\phi(f_1\cdots f_m)$ where all coefficients are in $A^m$. Since the coefficients $g_{a,b}$ are uniquely determined by $f$, the result follows.  

b) If we choose a homogeneous basis $\{f_\alpha\}$ of $A^m$ then $f_{\alpha}x^iy^j$ is a homogeneous basis of $A^m[x,y]$. We introduce a partial order on the latter by 
$$
f_{\alpha}x^iy^j\prec f_{\beta}x^{i'}y^{j'}\quad \mathrm{if}\ i\le i',j\le j', (i,j)\neq (i',j').
$$
We claim that $\phi(f_{\alpha})x^iy^j$ is another basis of $A^m[x,y]$. Indeed, by \eqref{eq: shifted f} it is related to $f_{\alpha}x^iy^j$ by a unitriangular matrix, therefore $\phi$ maps $A^m[x,y]$ isomorphically to itself.
\end{proof}



\begin{definition}
Let $I\subset \C[x_1,\ldots,x_n,y_1,\ldots,y_n,x,y]$ be the ideal 
$
I = \bigcap_{i=1}^n (x_i-x,y_i-y).$
\end{definition}

\begin{lemma}
\label{lem: f in I^k}
We have that
$
I^k=\bigcap_{i=1}^n (x_i-x,y_i-y)^k.
$
Furthermore,
$$f\in \spann\left(x_1^{a_1}y_1^{b_1} \cdots x_n^{a_n}y_n^{b_n}\mid a_i+b_i\ge k\quad \mathrm{for\ all}\ i\right)$$
if and only if $\phi(f)\in I^k$. 
\end{lemma}

\begin{proof}
First, observe that the monomial ideal
$
I'=\bigcap_{i=1}^n(x_i,y_i)  
$
is spanned by the monomials $x_1^{a_1}y_1^{b_1} \cdots x_n^{a_n}y_n^{b_n}$ such that $a_i+b_i\ge 1$ for all $i$. Similarly, 
$
(I')^k=\bigcap_{i=1}^n(x_i,y_i)^k 
$
is spanned by the monomials $x_1^{a_1}y_1^{b_1} \cdots x_n^{a_n}y_n^{b_n}$ such that $a_i+b_i\ge k$ for all $i$. It remains to notice that $\phi(I')=I$ and 
$$
\phi((I')^k)=I^k=\bigcap_{i=1}^n (x_i-x,y_i-y)^k.
$$
\end{proof}

\begin{lemma}\label{lem:changeofvariables}
    The map
    $\phi: A^{m+k}_{\geq k}[x,y] \rightarrow \C[x_1, y_1, \dots, x_n, y_n, x, y]$
    is injective with image $A^{m+k}[x,y] \cap I^k$.
\end{lemma}
\begin{proof}
    Lemma \ref{lem: A[x,y] span} tells us that $\phi$ maps $A^{m+k}[x,y]$ isomorphically to itself, and Lemma \ref{lem: f in I^k} tells us that $f \in A^{m+k}[x,y]$ satisfies the support condition $a_i+b_i\ge k$ if and only if $\phi(f) \in I^k$, so the result follows.
\end{proof}

\begin{theorem}
\label{thm: sections nested}
We have an isomorphism
$$
H^0(\Hilb^{n,n+1}(\C^2),\CO(m,k))\simeq A^{m+k}[x,y] \cap I^k.
$$
Equivalently, the global sections of $\CO(m,k)$ can be identified with  the $\sgn(m+k)$-component of $J^{m+k}\cap I^k\subseteq \C[x_1,y_1,\dots,x_n,y_n,x,y].$
\end{theorem}

\begin{proof}
We claim that Corollary \ref{cor:birationalequiv} can be strengthened to a commutative diagram of isomorphisms:
$$
\begin{tikzcd}
 H^0(\Hilb^n (\mathrm{Bl}_0 \C^2)\times \C^2,\CO(m+k)\otimes \CO(kE)_n) \arrow{r} \arrow{d}{\beta_n} &  A^{m+k}_{\ge k}[x,y] \arrow{d}{\mathrm{Id}}\\
 H^0(\Hilb^{n,n+1}_0(\C^2)\times \C^2,\CO(m,k)) \arrow{r} \arrow{d}{\alpha} & A^{m+k}_{\ge k}[x,y] \arrow{d}{\phi} \\
  H^0(\Hilb^{n,n+1}(\C^2),\CO(m,k))
  \arrow{r}
  & A^{m+k}[x,y]\cap I^k.
\end{tikzcd}
$$
The top horizontal arrow is given by Proposition \ref{prop: sections blowup}, where we regard $x,y$ as coordinates on the auxiliary $\C^2$. The map $\beta_n$ defines an isomorphism in Corollary \ref{cor:birationalequiv} and the map $\alpha$ is given by Lemma \ref{lem: factorization}. Finally, by Lemma \ref{lem:changeofvariables} $\phi$ is an isomorphism between $A^{m+k}_{\ge k}[x,y]$ and $A^{m+k}[x,y]\cap I^k$ and the result follows.
\end{proof}


\begin{example}
    When $k=1$ and $m=-1$, the sections of $\CO(-1,1) = \CL$ are identified with $I \cap A^0 = I^{S_n}$. In particular, we start with polynomials in $A^0_{\geq 1}$, which are polynomials in $\C[x_1, y_1, \dots, x_n, y_n]^{S_n}$ satisfying the support condition, and apply the map $\phi$.

    Let $\{m_S\}$ be the monomial basis of $\C[x_1, y_1, \dots, x_n, y_n]^{S_n}$ indexed by multisets $S \subseteq \Z^2_{\geq 0}$ of size $n$. The support condition is equivalent to the condition that $(0,0) \notin S$. Thus 
    $$\{\phi(m_S) | (0,0) \notin S\}$$
    gives a basis for global sections of $\CL$ coming from sections on the Hilbert scheme of the blowup.
\end{example}

\begin{example}
    When $k=1$ and $m=0$, sections of $\CO(0,1) = \CO(1) \otimes \CL$ are identified with $I \cap A = I^{\sgn}$. We can similarly take the determinant basis $\{\Delta_S\}$ of antisymmetric polynomials, indexed by subsets $S \subseteq \Z^2_{\geq 0}$ of size $n$. Again the support condition tells us that $S$ does not contain the origin, and 
    $$\{\phi(\Delta_S) | (0,0) \notin S \}$$ 
    gives a basis for global sections.
\end{example}

\section{Trailing Terms of Global Sections}

\subsection{Trailing Terms Analysis}

In this section we describe bases for the global sections of line bundles on $\Hilb^n(\Bl_0 \C^2)$ and $\Hilb^{n,n+1}(\C^2)$ by characterizing the sets of trailing term exponents of the corresponding polynomials. Throughout, we use the lexicographic term order with $x_1<\cdots<x_n<y_1<\cdots<y_n$. Here the \textit{trailing term} of $f$ is defined to be the term $x_1^{a_1}\cdots x_n^{a_n}y_1^{b_1}\cdots y_n^{b^n}$ such that $(a_1,\dots,a_n,b_1,\dots,b_n)$ is lexicographically minimal among all the terms of $f$. 

\begin{example}\label{ex:m=0}
    For $m=0$, $A^0_{\geq k}$ is the set of symmetric polynomials with respect to the diagonal $S_n$-action on $\C[x_1,y_1,\dots,x_n,y_n]$ satisfying the support constraint. There is a simple basis for this space consisting of polynomials of the form
    \[ x_1^{a_1}\cdots x_n^{a_n}y_1^{b_1}\cdots y_n^{b_n} + (\text{symmetric terms}),\]
    where $(a_1,b_1)\leq \cdots \leq (a_n,b_n)$ are integer points with $a_i,b_i\geq 0$ and $a_i+b_i\geq k$. With the points $(a_i,b_i)$ ordered in non-decreasing lexicographic order as above, $x_1^{a_1}\cdots x_n^{a_n}y_1^{b_1}\cdots y_n^{b_n}$ is the trailing term of this polynomial. One can show that the trailing term of any element of $A^0_{\geq k}$ is of this form.
\end{example}

\begin{example}\label{ex:m=1}
    For $m=1$, $A^1_{\geq k}$ is the set of anti-symmetric polynomials with respect to the diagonal $S_n$-action on $\C[x_1,y_1,\dots,x_n,y_n]$ satisfying the support constraint. There is again a simple basis for this space, the determinants $\Delta_S$ where $S = \{(a_1,b_1),\dots,(a_n,b_n)\}$ is a collection of distinct integer points $(a_1,b_1)<\cdots<(a_n,b_n)$ with $a_i,b_i\geq 0$ and $a_i+b_i\geq k$. Again, we order $(a_i,b_i)$ in increasing lexicographic order as above so that $x_1^{a_1}\cdots x_n^{a_n}y_1^{b_1}\cdots y_n^{b_n}$ is the trailing term of this polynomial. Similarly, one can show that the trailing term of any element of $A^1_{\geq k}$ is of this form.
\end{example}

Characterizing the trailing terms of $A^m_{\geq k}$ for $m>1$ is more difficult as there are no longer obvious bases to work with. As we will show, the trailing term exponents of $A^m_{\geq k}$ are characterized by the following explicitly defined set.

\begin{definition}\label{def:P(m,k)}
    For any integers $m,k\geq0$, let $\mathcal{P}(m,k)\subseteq\Z_{\ge 0}^{2n}$ be the subset defined by:
    \begin{enumerate}
        \item $0\leq a_1\leq a_2 \leq \cdots \leq a_n$,
        \item for any $j=1,\dots,n-1$ for which $a_j = a_{j+1}$, we have $b_{j+1}\geq b_j+m$, and
        \item for each $j=1,\dots,n,$ we have $b_j\geq \max\{k-a_j,0\} + \sum_{i=1}^{j-1}\max\{m-(a_j-a_i),0\}.$
    \end{enumerate}
\end{definition}

\begin{theorem}\label{thm:blowuptrailingterms}
    For any $n\geq 1$ and $m,k\geq 0$, a monomial $x_1^{a_1}y_1^{b_1}\cdots x_n^{a_n}y_n^{b_n}$ is the trailing term of some polynomial $f \in A^m_{\geq k}$ if and only if $(a_1,\dots,a_n,b_1,\dots,b_n)\in \mathcal{P}(m,k)$.
    For $m>0$, these are exactly the monomials that appear as the trailing term of an $m$-fold product of determinants $\Delta_{S^{(1)}}\cdots \Delta_{S^{(m)}}$ for some $n$-element subsets $S^{(1)},\dots,S^{(m)}\subseteq \Z^2_{\geq 0}$ such that the product $\Delta_{S^{(1)}}\cdots \Delta_{S^{(m)}}$ satisfies the support condition defining $A^m_{\geq k}$.
\end{theorem}

The proof of Theorem \ref{thm:blowuptrailingterms} is given below after some preparatory results. \\

By analogy with the interpretation of the trailing terms in Examples \ref{ex:m=0} and \ref{ex:m=1}, we identify a point $(a_1,\dots,a_n,b_1,\dots,b_n)\in \Z^{2n}$ with an ordered $n$-tuple of points $(a_1,b_1),\dots,(a_n,b_n)$. We refer to the parameters $m$ and $k$ in Definition \ref{def:P(m,k)} as the \textit{separation parameter} and the \textit{support parameter} respectively. For any subset $S\subseteq \Z^2_{\geq 0}$ written $S = \{(a_1,b_1),\dots,(a_n,b_n)\}$ we will always assume without loss of generality that the points are labeled in increasing lexicographic order. 

\begin{definition}
    For sets $S^{(1)},\dots,S^{(m)}\subseteq \Z^2_{\geq 0}$, labeled $S^{(j)} = \{ p^{(j)}_1,\dots,p_n^{(j)}\}$, let $S^{(1)}+\dots+S^{(m)}$ denote the ordered set of points $\{p_1,\dots,p_n\}$ where $p_i = p_i^{(1)}+\cdots+p_i^{(m)}$ for each $i$ coordinate-wise.
\end{definition}

For $m>1$, $A^m$ is spanned by products $\Delta_{S^{(1)}}\cdots \Delta_{S^{(m)}}$, and the trailing term exponent of the product is $S^{(1)}+\dots+S^{(m)}$. In \cite{C} (Propositions 3.7 and 3.9) it was shown that these are the only trailing terms that appear among elements of $A^m$, and that $x_1^{a_1}\cdots x_n^{a_n}y_1^{b_1}\cdots y_n^{b_n}$ is one of these trailing terms if and only if $(a_1,\dots,a_n,b_1,\dots,b_n)\in \mathcal{P}(m,0)$.

\begin{example}\label{ex:c2pointdecomp}
    Consider the set of points $S = \{p_1,\dots,p_4\} = \{(0,0),(0,2),(1,2),(2,1)\}$ with coordinates $(a_1,b_1),\dots,(a_4,b_4)$. This corresponds to a point in $\mathcal{P}(2,0)$, so by the results of \cite{C} there is a product of determinants $\Delta_{S^{(1)}}\Delta_{S^{(2)}}\in A^2$ with trailing term $y_2^2 x_3 y_3^2 x_4^2 y_4$. Equivalently, there exist subsets $S^{(1)} = \{ p^{(1)}_1,p^{(1)}_2,p^{(1)}_3,p^{(1)}_4\}$ and $S^{(2)} = \{ p^{(2)}_1,p^{(2)}_2,p^{(2)}_3,p^{(2)}_4\}$ of $\Z^2_{\geq 0}$ such that $S = S^{(1)}+S^{(2)}$. Such a decomposition is $S^{(1)} = \{ (0,0),(0,1),(0,2),(1,0)\}$ and $S^{(2)} = \{ (0,0),(0,1),(1,0),(1,1) \}$. On the other hand $S' = \{p_1,\dots,p_4\} = \{(0,0),(0,2),(1,1),(2,1)\}$ does not satisfy the conditions for $m=2$, since $b_3 < \max\{2-a_1,0\}+\max\{2-a_2,0\} = 2$. The theorem asserts that there is no element of $A^2$ with trailing term $y_2^2 x_3 y_3 x_4^2 y_4$.
    
    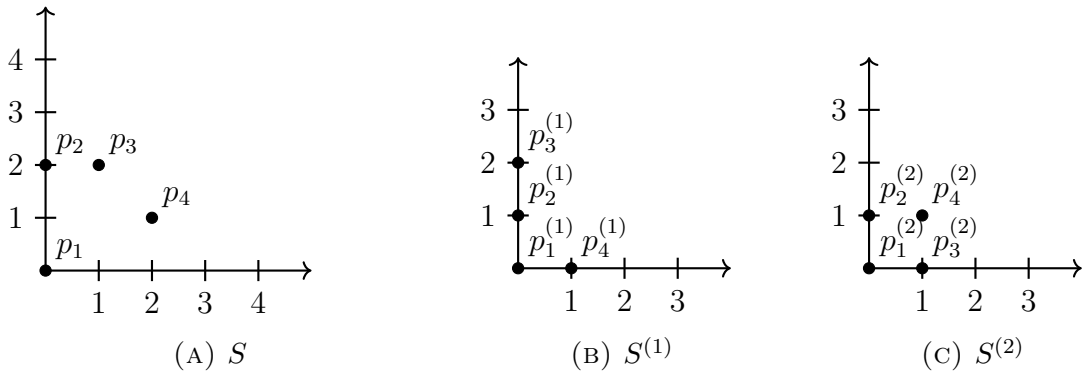
\begin{figure}[h]
        \centering
    \begin{subfigure}{.35\textwidth}
        \begin{tikzpicture}[scale = .7]
            \draw[<->, thick] (5,0)--(0,0)--(0,5);
            \draw[thick] (1,.2)--(1,-.2) node[below] {$1$};
            \draw[thick] (2,.2)--(2,-.2) node[below] {$2$};
            \draw[thick] (3,.2)--(3,-.2) node[below] {$3$};
            \draw[thick] (4,.2)--(4,-.2) node[below] {$4$};
            \draw[thick] (.2,1)--(-.2,1) node[left] {$1$};
            \draw[thick] (.2,2)--(-.2,2) node[left] {$2$};
            \draw[thick] (.2,3)--(-.2,3) node[left] {$3$};
            \draw[thick] (.2,4)--(-.2,4) node[left] {$4$};
            \filldraw[black] (0,0) circle (3pt) node[above right]{$p_1$};
            \filldraw[black] (0,2) circle (3pt) node[above right]{$p_2$};
            \filldraw[black] (1,2) circle (3pt) node[above right]{$p_3$};
            \filldraw[black] (2,1) circle (3pt) node[above right]{$p_4$};
        \end{tikzpicture}
        \caption{$S$}
            \end{subfigure}%
    \hspace{.03\textwidth}
    \begin{subfigure}{.25\textwidth}
        \begin{tikzpicture}[scale = .7]
            \draw[<->, thick] (4,0)--(0,0)--(0,4);
            \draw[thick] (1,.2)--(1,-.2) node[below] {$1$};
            \draw[thick] (2,.2)--(2,-.2) node[below] {$2$};
            \draw[thick] (3,.2)--(3,-.2) node[below] {$3$};
            \draw[thick] (.2,1)--(-.2,1) node[left] {$1$};
            \draw[thick] (.2,2)--(-.2,2) node[left] {$2$};
            \draw[thick] (.2,3)--(-.2,3) node[left] {$3$};
            \filldraw[black] (0,0) circle (3pt) node[above right]{$p_1^{(1)}$};
            \filldraw[black] (0,1) circle (3pt) node[above right]{$p_2^{(1)}$};
            \filldraw[black] (0,2) circle (3pt) node[above right]{$p_3^{(1)}$};
            \filldraw[black] (1,0) circle (3pt) node[above right]{$p_4^{(1)}$};
        \end{tikzpicture}
        \caption{$S^{(1)}$}
            \end{subfigure}%
    \hspace{.03\textwidth}
    \begin{subfigure}{.25\textwidth}
        \begin{tikzpicture}[scale = .7]
            \draw[<->, thick] (4,0)--(0,0)--(0,4);
            \draw[thick] (1,.2)--(1,-.2) node[below] {$1$};
            \draw[thick] (2,.2)--(2,-.2) node[below] {$2$};
            \draw[thick] (3,.2)--(3,-.2) node[below] {$3$};
            \draw[thick] (.2,1)--(-.2,1) node[left] {$1$};
            \draw[thick] (.2,2)--(-.2,2) node[left] {$2$};
            \draw[thick] (.2,3)--(-.2,3) node[left] {$3$};
            \filldraw[black] (0,0) circle (3pt) node[above right]{$p_1^{(2)}$};
            \filldraw[black] (0,1) circle (3pt) node[above right]{$p_2^{(2)}$};
            \filldraw[black] (1,0) circle (3pt) node[above right]{$p_3^{(2)}$};
            \filldraw[black] (1,1) circle (3pt) node[above right]{$p_4^{(2)}$};
        \end{tikzpicture}
                \caption{$S^{(2)}$}
            \end{subfigure}
        \caption{Quadruples of points satisfying $S=S^{(1)}+S^{(2)}$, so that $S$ is the trailing term exponent of an element of $A^2$.}
    \end{figure}
\end{example}

Much of the proof of Theorem \ref{thm:blowuptrailingterms} is based on intermediate results of \cite{C} used to prove the $k=0$ case. The new combinatorics needed for the proof can be phrased in terms of the following operation on integer subsets.

\begin{definition}
    For any integer $k\geq 0$ and subset $S\subseteq \Z^2_{\geq 0}$, the \textit{$k$-lift of $S$} is the set $\ell_k(S) = \{ (a,b+\max\{k-a,0\})\, | \, (a,b)\in S \}.$
\end{definition}

Inspecting the definitions, one immediately sees the following.

\begin{lemma}
The $k$-lift operation $\ell_k$ is a bijection from $\mathcal{P}(m,0)$ to $\mathcal{P}(m,k)$.
\end{lemma}

The main combinatorial fact needed for the proof of Theorem \ref{thm:blowuptrailingterms} is the following description of the $k$-lift of a sum $S^{(1)}+\cdots+S^{(m)}$.

\begin{lemma}\label{lem:liftdecomp}
    Let $S^{(1)},\dots,S^{(m)}\subseteq \Z^2_{\geq 0}$ be subsets of size $n$ and set $S = S^{(1)}+\cdots+S^{(m)}$. For any integer $k\geq 0$, there exist $k_1,\dots, k_m\geq 0$ with $\sum_i k_i =k$ such that $\ell_k(S) = \ell_{k_1}(S^{(1)})+\cdots+\ell_{k_m}(S^{(m)})$. Furthermore, we may choose the $k_1,\dots,k_m$ such that for any $j=1,\dots,n$ with $a_j\leq k$, we have $a_j^{(i)}\leq k_i$ for all $i=1,\dots,m$, where $a_1,\dots,a_n$ are the first-coordinates of the points in $S$ and $a_1^{(i)},\dots,a_n^{(i)}$ are those for $S^{(i)}$.
\end{lemma}

We will not need the final condition in the previous lemma for its applications, but including it in the statement helps to simplify the proof. 

\begin{proof}
    We proceed by induction on $k$ with the trivial $k=0$ base case. Given $\ell_k(S) = \ell_{k_1}(S^{(1)})+\cdots+\ell_{k_m}(S^{(m)})$ with $k_1,\dots,k_n$ as in the statement of the lemma, we will show that one of the $k_i$'s can be replaced by $k_i+1$ to obtain such a decomposition for $\ell_{k+1}(S)$. Writing $\ell_k(S) = \{(a_1,b_1),\dots,(a_n,b_n)\}$, we have $\ell_{k+1}(S) = \{(a_1,b_1'),\dots,(a_n,b_n')\}$, where
    \[ b_j' = \begin{cases}
        b_j+1 & \text{if } a_j\leq k,\\
        b_j & \text{if } a_j>k.
    \end{cases} \]
    If $k\geq a_1,\dots,a_n$, then all the points in $\ell_k(S)$ are shifted up by one to obtain the corresponding points in $\ell_{k+1}(S)$. In this case, we may replace any one of the $k_i$'s with $k_i+1$. Indeed, by hypothesis we have $a^{(i)}\leq k_i$ for all $i,j$, so replacing $\ell_{k_i}(S^{(i)})$ by $\ell_{k_i+1}(S^{(i)})$ in the decomposition also shifts all the points up by one as desired.

    Otherwise let $j$ be the smallest index such that $a_j>k$. In other words, $(a_1,b_1),\dots,(a_{j-1},b_{j-1})$ are the points of $\ell_k(S)$ that must be shifted up to obtain the corresponding points of $\ell_{k+1}(S)$, while $(a_j,b_j),\dots,(a_n,b_n)$ are unchanged. Since $a_j = a_j^{(1)}+\cdots+a_j^{(m)}$ and $k = k_1+\cdots+k_m$, there is some index $i$ such that $a_j^{(i)}>k_i$. By our choice of $j$ we have $a_1,\dots,a_{j-1}\leq k$ so by the final induction hypothesis we also have $a_1^{(i)},\dots,a_{j-1}^{(i)}\leq k_i$. Therefore, replacing $\ell_{k_i}(S^{(i)})$ by $\ell_{k_i+1}(S^{(i)})$ in the decomposition also shifts up exactly the points $(a_1,b_1),\dots,(a_{j-1},b_{j-1})$ by one, as desired. 
    
    Lastly, we check that the final induction hypothesis is preserved by this procedure. In each step, for any index $j'$ such that $a_{j'}>k$ (i.e. $j'>j$ in the notation of the previous paragraph), we are increasing one of the $k_i$'s that has $a_{j'}^{(i)}\geq a_{j}^{(i)}>k_i$. It follows that for any index $j'$ such that $a_{j'}\leq k$, we must have $a_{j'}^{(i)}\leq k_i$ for all $i$.
\end{proof}

\begin{example}\label{ex:blowuppointdecomp}
    Let $S = S^{(1)}+S^{(2)}$ be the decomposition of the set of points studied in Example \ref{ex:c2pointdecomp} corresponding to trailing term of the product $\Delta_{S^{(1)}}\Delta_{S^{(2)}}\in A^2$. For $k=1,2$, the unique such decompositions of $\ell_k(S)$ are $\ell_1(S) = S^{(1)}+\ell_1(S^{(2)})$ and $\ell_2(S) = \ell_1(S^{(1)})+\ell_1(S^{(2)})$. For $k\geq 2$, any choice of $k_1+k_2 = k$ with $k_1,k_2\geq 1$ gives $\ell_k(S) = \ell_{k_1}(S^{(1)})+\ell_{k_2}(S^{(2)})$

    \begin{figure}[h]
    \label{fig:decomp1modified}
        \centering
    \begin{subfigure}{.35\textwidth}
        \begin{tikzpicture}[scale = .7]
            \draw[<->, thick] (5,0)--(0,0)--(0,5);
            \draw[thick] (1,.2)--(1,-.2) node[below] {$1$};
            \draw[thick] (2,.2)--(2,-.2) node[below] {$2$};
            \draw[thick] (3,.2)--(3,-.2) node[below] {$3$};
            \draw[thick] (4,.2)--(4,-.2) node[below] {$4$};
            \draw[thick] (.2,1)--(-.2,1) node[left] {$1$};
            \draw[thick] (.2,2)--(-.2,2) node[left] {$2$};
            \draw[thick] (.2,3)--(-.2,3) node[left] {$3$};
            \draw[thick] (.2,4)--(-.2,4) node[left] {$4$};
            \filldraw[opacity = .4] (0,0)--(0,1)--(1,0)--cycle;
            \filldraw[black] (0,1) circle (3pt) node[above right]{$p_1'$};
            \filldraw[black] (0,3) circle (3pt) node[above right]{$p_2'$};
            \filldraw[black] (1,2) circle (3pt) node[above right]{$p_3$};
            \filldraw[black] (2,1) circle (3pt) node[above right]{$p_4$};
        \end{tikzpicture}
        \caption{$\ell_1(S) = S^{(1)}+\ell_1(S^{(2)})$}
            \end{subfigure}%
    \hspace{.03\textwidth}
    \begin{subfigure}{.25\textwidth}
        \begin{tikzpicture}[scale = .7]
            \draw[<->, thick] (4,0)--(0,0)--(0,4);
            \draw[thick] (1,.2)--(1,-.2) node[below] {$1$};
            \draw[thick] (2,.2)--(2,-.2) node[below] {$2$};
            \draw[thick] (3,.2)--(3,-.2) node[below] {$3$};
            \draw[thick] (.2,1)--(-.2,1) node[left] {$1$};
            \draw[thick] (.2,2)--(-.2,2) node[left] {$2$};
            \draw[thick] (.2,3)--(-.2,3) node[left] {$3$};
            \filldraw[black] (0,0) circle (3pt) node[above right]{$p_1^{(1)}$};
            \filldraw[black] (0,1) circle (3pt) node[above right]{$p_2^{(1)}$};
            \filldraw[black] (0,2) circle (3pt) node[above right]{$p_3^{(1)}$};
            \filldraw[black] (1,0) circle (3pt) node[above right]{$p_4^{(1)}$};
        \end{tikzpicture}
        \caption{$S^{(1)}$}
            \end{subfigure}%
    \hspace{.03\textwidth}
    \begin{subfigure}{.25\textwidth}
        \begin{tikzpicture}[scale = .7]
            \draw[<->, thick] (4,0)--(0,0)--(0,4);
            \draw[thick] (1,.2)--(1,-.2) node[below] {$1$};
            \draw[thick] (2,.2)--(2,-.2) node[below] {$2$};
            \draw[thick] (3,.2)--(3,-.2) node[below] {$3$};
            \draw[thick] (.2,1)--(-.2,1) node[left] {$1$};
            \draw[thick] (.2,2)--(-.2,2) node[left] {$2$};
            \draw[thick] (.2,3)--(-.2,3) node[left] {$3$};
            \filldraw[opacity = .4] (0,0)--(0,1)--(1,0)--cycle;
            \filldraw[black] (0,1) circle (3pt) node[above right]{${p_1'}^{(2)}$};
            \filldraw[black] (0,2) circle (3pt) node[above right]{${p_2'}^{(2)}$};
            \filldraw[black] (1,0) circle (3pt) node[above right]{${p_3}^{(2)}$};
            \filldraw[black] (1,1) circle (3pt) node[above right]{${p_4}^{(2)}$};
        \end{tikzpicture}
                \caption{$\ell_1(S^{(2)})$}
            \end{subfigure}
    
    \begin{subfigure}{.35\textwidth}
        \begin{tikzpicture}[scale = .7]
            \draw[<->, thick] (5,0)--(0,0)--(0,5);
            \draw[thick] (1,.2)--(1,-.2) node[below] {$1$};
            \draw[thick] (2,.2)--(2,-.2) node[below] {$2$};
            \draw[thick] (3,.2)--(3,-.2) node[below] {$3$};
            \draw[thick] (4,.2)--(4,-.2) node[below] {$4$};
            \draw[thick] (.2,1)--(-.2,1) node[left] {$1$};
            \draw[thick] (.2,2)--(-.2,2) node[left] {$2$};
            \draw[thick] (.2,3)--(-.2,3) node[left] {$3$};
            \draw[thick] (.2,4)--(-.2,4) node[left] {$4$};
            \filldraw[opacity = .4] (0,0)--(0,2)--(2,0)--cycle;
            \filldraw[black] (0,2) circle (3pt) node[above right]{$p_1''$};
            \filldraw[black] (0,4) circle (3pt) node[above right]{$p_2''$};
            \filldraw[black] (1,3) circle (3pt) node[above right]{$p_3'$};
            \filldraw[black] (2,1) circle (3pt) node[above right]{$p_4$};
        \end{tikzpicture}
        \caption{$\ell_2(S) = \ell_1(S^{(1)})+\ell_1(S^{(2)})$}
            \end{subfigure}%
    \hspace{.03\textwidth}
    \begin{subfigure}{.25\textwidth}
        \begin{tikzpicture}[scale = .7]
            \draw[<->, thick] (4,0)--(0,0)--(0,4);
            \draw[thick] (1,.2)--(1,-.2) node[below] {$1$};
            \draw[thick] (2,.2)--(2,-.2) node[below] {$2$};
            \draw[thick] (3,.2)--(3,-.2) node[below] {$3$};
            \draw[thick] (.2,1)--(-.2,1) node[left] {$1$};
            \draw[thick] (.2,2)--(-.2,2) node[left] {$2$};
            \draw[thick] (.2,3)--(-.2,3) node[left] {$3$};
            \filldraw[opacity = .4] (0,0)--(0,1)--(1,0)--cycle;
            \filldraw[black] (0,1) circle (3pt) node[above right]{${p_1'}^{(1)}$};
            \filldraw[black] (0,2) circle (3pt) node[above right]{${p_2'}^{(1)}$};
            \filldraw[black] (0,3) circle (3pt) node[above right]{${p_3'}^{(1)}$};
            \filldraw[black] (1,0) circle (3pt) node[above right]{${p_4}^{(1)}$};
        \end{tikzpicture}
        \caption{$\ell_1(S^{(1)})$}
            \end{subfigure}%
    \hspace{.03\textwidth}
    \begin{subfigure}{.25\textwidth}
        \begin{tikzpicture}[scale = .7]
            \draw[<->, thick] (4,0)--(0,0)--(0,4);
            \draw[thick] (1,.2)--(1,-.2) node[below] {$1$};
            \draw[thick] (2,.2)--(2,-.2) node[below] {$2$};
            \draw[thick] (3,.2)--(3,-.2) node[below] {$3$};
            \draw[thick] (.2,1)--(-.2,1) node[left] {$1$};
            \draw[thick] (.2,2)--(-.2,2) node[left] {$2$};
            \draw[thick] (.2,3)--(-.2,3) node[left] {$3$};
            \filldraw[opacity = .4] (0,0)--(0,1)--(1,0)--cycle;
            \filldraw[black] (0,1) circle (3pt) node[above right]{${p_1'}^{(2)}$};
            \filldraw[black] (0,2) circle (3pt) node[above right]{${p_2'}^{(2)}$};
            \filldraw[black] (1,0) circle (3pt) node[above right]{${p_3}^{(2)}$};
            \filldraw[black] (1,1) circle (3pt) node[above right]{${p_4}^{(2)}$};
        \end{tikzpicture}
                \caption{$\ell_1(S^{(2)})$}
            \end{subfigure}
        \caption{The decompositions of $\ell_1(S)$ and $\ell_2(S)$ obtained by modifying the decomposition $S = S^{(1)}+S^{(2)}$ from Example \ref{ex:c2pointdecomp}.}
    \end{figure}
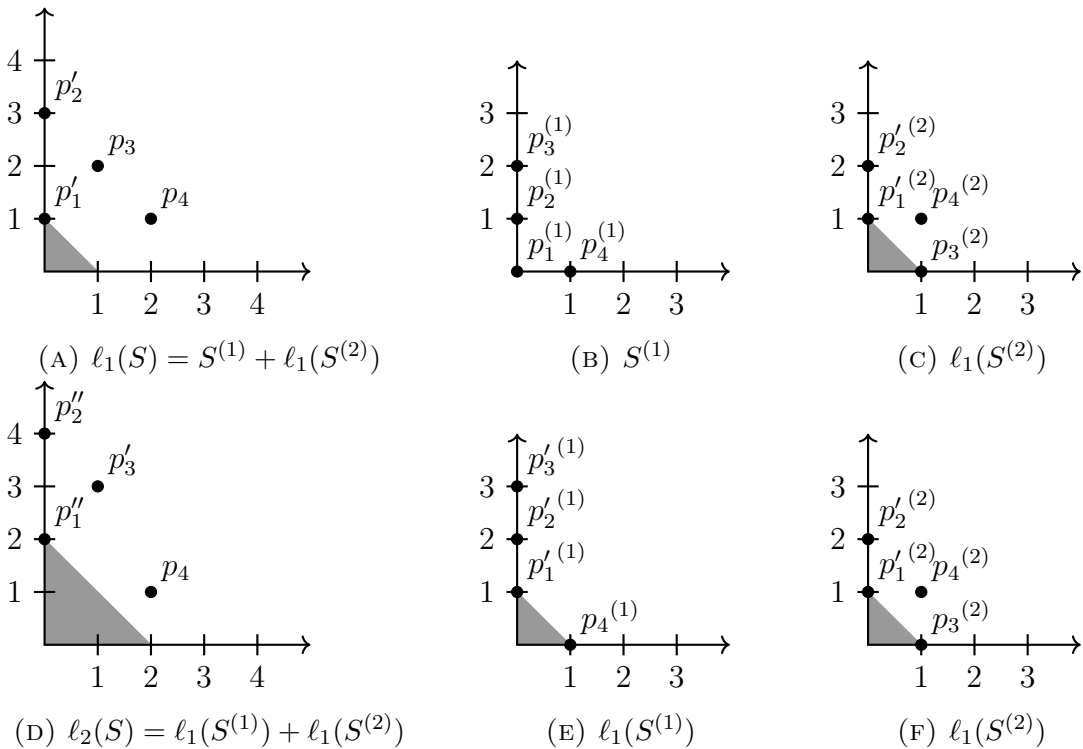
\end{example}

\begin{proof}[Proof of Theorem \ref{thm:blowuptrailingterms}]
    First we show that for $f\in A^m$ a polynomial with trailing term $x_1^{a_1}y_1^{b_1}\cdots x_n^{a_n}y_n^{b_n}$ we must have $(a_1,\dots,a_n,b_1,\dots,b_n)\in \mathcal{P}(m,k)$. We argue directly that $a_1\leq \cdots\leq a_n$ is satisfied. Indeed, since $f$ is either symmetric or anti-symmetric, terms corresponding to the permutations of these pairs of exponents also appear in $f$. For $x_1^{a_1}y_1^{b_1}\cdots x_n^{a_n}y_n^{b_n}$ to be the trailing term in this lexicographic term order, it must in particular be the last term among these permutations, which implies $0\leq a_1\leq a_2 \leq \cdots \leq a_n$. 
    
    To see that conditions (2) and (3) in the definition of $\mathcal{P}(m,k)$ are necessary, we apply some intermediate results of \cite{C}. Specifically, since $f\in A^m$ we may apply Proposition 3.5 of \cite{C}, noting that in the notation of this proposition the hypothesis $f\in \overline{A}^m$ is satisfied since by Corollary 3.10 of \cite{C} we have $A^m = \overline{A}^m$. This proposition directly states that our condition (2) is satisfied, and also that for any $j=1,\dots,n$, the Newton polytope of $f$ contains some integer point corresponding to exponents $a_j$ and $b_j' := b_j-\sum_{i=1}^{j-1} \max\{ m-(a_j-a_i),0\}$ on $x_j$ and $y_j$ respectively. But we are assuming that $f\in A^m_{\geq k}$, so these integer points $(a_j,b_j')$ must satisfy the inequality restriction on the support of $f$. In other words, we have $b_j' \geq \max\{k-a_j',0\}$, which is equivalent to the inequality asserted in condition (3). This shows that conditions (1), (2), and (3) are necessary for trailing term exponents of elements of $A^m_{\geq k}$.

    For the other direction, as well as the additional claims, recall that we have observed that every element of $\mathcal{P}(m,k)$ is of the form $\ell_k(S)$ for some set $S= \{ (a_1,b_1),\dots,(a_n,b_n) \}$ corresponding to a point of $\mathcal{P}(m,0)$. By Theorem \cite{C} Propositions 3.7 and 3.9 we can decompose $S= S^{(1)}+\cdots+S^{(m)}$ for some sets $S^{(1)},\dots,S^{(m)}$, and by Lemma \ref{lem:liftdecomp} there are some $k_1,\dots,k_m$ with $\sum k_i = k$ such that $\ell_k(S) = \ell_{k_1}(S^{(1)})+\cdots+\ell_{k_m}(S^{(m)})$. This means that $\ell_k(S)$ is obtained as the trailing term exponent of the product
    \[ f = \Delta_{\ell_{k_1}(S^{(1)})}\cdots \Delta_{\ell_{k_m}(S^{(m)})}. \]
    The final observation is that $\Delta_{\ell_{k_i}(S^{(i)})}\in A^1_{\geq k_i}$ for each $i=1,\dots,m$, since by the definition of the lifting operator every point $(a,b)\in \ell_{k_i}(S^{(i)})$ satisfies $a+b\geq k$. This implies that their product $f$ is in $A^m_{\geq k}$, exhibiting $\ell_k(S)$ as the trailing term exponent of a product as claimed.
\end{proof}

\begin{corollary}\label{cor:surjectivity}
For $m,k>0$, the natural multiplication map
$$
\bigoplus_{k_1+\ldots+k_m=k}\bigotimes_{i=1}^{m}H^0(\Hilb^{n,n+1}(\C^2),\CO(1-k_i,k_i))\to H^0(\Hilb^{n,n+1}(\C^2),\CO(m-k,k))
$$
is surjective. In other words, the bigraded algebra 
$$
\bigoplus_{k,m=1}^{\infty}H^0(\Hilb^{n,n+1}(\C^2),\CO(m-k,k))=\bigoplus_{k,m=1}^{\infty}H^0(\Hilb^{n,n+1}(\C^2),\pi_n^*\CO(m)\otimes \CL^k)
$$
is generated by the components of $m$-degree 1. 
\end{corollary}

Using Theorem \ref{thm:blowuptrailingterms}, we can determine the trailing terms of $A^{m+k}[x,y]\cap I^k$ with respect to the extended lexicographic term order with $x<y<x_1<\cdots<x_n<y_1<\cdots<y_n$. Indeed, combining \eqref{eq: shifted f} with Theorem \ref{thm:blowuptrailingterms} we obtain the following:

\begin{corollary}\label{cor:trailingterms}
        For any $n\geq 1$ and $m,k\geq 0$, a monomial $x^a y^b x_1^{a_1}y_1^{b_1}\cdots x_n^{a_n}y_n^{b_n}$ is the trailing term of some polynomial $g\in A^{m+k}[x,y]\cap I^k \subseteq \C[x,y,x_1,y_1,\dots,x_n,y_n]$ if and only if $a,b\geq 0$ and $(a_1,\dots,a_n,b_1,\dots,b_n)\in \mathcal{P}(m+k,k)$.
\end{corollary}

\subsection{Character Formulas}\label{sec: character formulas}

In this section, we give combinatorial formulas for the characters of spaces of global sections of line bundles on the nested Hilbert scheme. These can be computed by localization formulas from Corollary \ref{cor: localization}, alternatively, one can use Corollary \ref{cor:trailingterms} to obtain the following result. 

\begin{corollary}
\label{cor: Hilbert series}
The Hilbert series of $H^0\left(\Hilb^{n,n+1}(\C^2),\CO(m,k)\right)$ is given by 
$$
H_{m,k}(q,t)=\frac{1}{(1-q)(1-t)}\sum_{\mathcal{P}(m+k,k)}q^{a_1+\ldots+a_n}t^{b_1+\ldots+b_n}.
$$
\end{corollary}

As written, Corollary \ref{cor: Hilbert series} involves a weighted sum over all integer points in the unbounded set $\mathcal{P}(m+k,k)$. However, for any given $m$ and $k$ one can group the points of $\mathcal{P}(m+k,k)$ into a finite number of subsets, for each of which the weighted sum is a rational function. This therefore gives an expression for the Hilbert series $H_{m,k}(q,t)$ as a rational function.

\begin{example}
\label{ex:k=m=1}
    We compute the Hilbert series $H_{1,1}(q,t)$ in the case $n=2$. Using Corollary \ref{cor: Hilbert series}, we will take the weighted sum over $\mathcal{P}(2,1)$, the set of integer points $(a_1,a_2,b_1,b_2)$ satisfying
    \begin{enumerate}
        \item $0\leq a_1\leq a_2$,
        \item if $a_1 = a_2$ then $b_2\geq b_1+2$, and 
        \item $b_1\geq \max\{1-a_1,0\}$ and $b_2\geq \max\{ 1-a_2,0 \}+\max\{2-(a_2-a_1),0\}$.
    \end{enumerate}
    We partition $\mathcal{P}(2,1)$ into two cases, each of which has three subcases.
    \begin{enumerate}
        \item $a_1=0$:
        \begin{enumerate}
            \item $a_2=a_1=0$: Here we have $b_1\geq 1$ and $b_2\geq b_1+2$. These conditions give a two-dimensional cone with vertex $(0,0,1,3)$ and ray generators $(0,0,1,1)$ and $(0,0,0,1)$, so the weighted sum of these points is $ \frac{t^4}{(1-t)(1-t^2)}$.
            \item $a_2=a_1+1=1$: Here we have $b_1\geq 1$ and $b_2\geq 1$. These conditions give a two-dimensional cone with vertex $(0,1,1,1)$ and ray generators $(0,0,1,0)$ and $(0,0,0,1)$, so the weighted sum of these points is $\frac{qt^2}{(1-t)^2}$.
            \item $a_2\geq a_1 +2=2$: Here we have $b_1\geq 1$ and $b_2\geq 0$. These conditions give a three-dimensional cone with vertex $(0,2,1,0)$ and ray generators $(0,1,0,0),$ $(0,0,1,0)$, and $(0,0,0,1)$, so the weighted sum of these points is $\frac{q^2t}{(1-q)(1-t)^2}$.
        \end{enumerate}
        \item $a_1\geq 1$:
        \begin{enumerate}
            \item $a_2= a_1$: Here we have $b_1\geq 0$ and $b_2\geq b_1+2$. These conditions give a three-dimensional cone with vertex $(1,1,0,2)$ and ray generators $(1,1,0,0),$ $(0,0,1,1)$, and $(0,0,0,1)$, so the weighted sum of these points is $\frac{q^2t^2}{(1-q^2)(1-t)(1-t^2)}$.
            \item $a_2=a_1+1$: Here we have $b_1\geq 0$ and $b_2\geq 1$. These conditions give a three-dimensional cone with vertex $(1,2,0,1)$ and ray generators $(1,1,0,0)$, $(0,0,1,0)$, and $(0,0,0,1)$, so the weighted sum of these points is $\frac{q^3t}{(1-q^2)(1-t)^2}$.
            \item $a_2\geq a_1+2$: Here we have $b_1\geq 0$ and $b_2 \geq 0$. These conditions give a four-dimensional cone with vertex $(1,3,0,0)$ and ray generators $(1,1,0,0)$, $(0,1,0,0)$, $(0,0,1,0)$, and $(0,0,0,1)$, so the weighted sum of these points is $\frac{q^4}{(1-q)(1-q^2)(1-t)^2}$.
        \end{enumerate}
    \end{enumerate}
    The weighted sum over all points in $\mathcal{P}(2,1)$ is the sum of the six rational functions above:
    \[ \frac{q^4t^3 + q^3t^4 - q^4t^2 - q^3t^3 - q^2t^4 - q^4t - q^3t^2 - q^2t^3 - qt^4 + q^4 + q^3t + q^2t^2 + qt^3 + t^4 + q^2t + qt^2}{(1-q)(1-q^2)(1-t)(1-t^2)}.\]
    By Corollary \ref{cor: Hilbert series}, the Hilbert series $H_{1,1}(q,t)$ is $\frac{1}{(1-t)(1-q)}$ times the above rational function.  One can check that by substituting $k=m=1$ into the localization formula in Example \ref{ex: n=2 localization}, one recovers the same answer.
    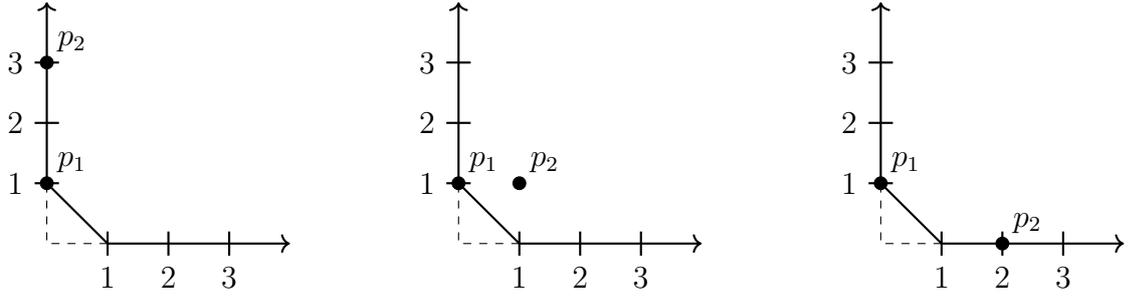
\begin{figure}[h]
    \begin{subfigure}{.3\textwidth}
    \centering
    \begin{tikzpicture}[scale = .8]
        \draw[<->, thick] (4,0)--(1,0)--(0,1)--(0,4);
        \draw[dashed] (1,0)--(0,0)--(0,1);
        \draw[thick] (1,.2)--(1,-.2) node[below] {$1$};
        \draw[thick] (2,.2)--(2,-.2) node[below] {$2$};
        \draw[thick] (3,.2)--(3,-.2) node[below] {$3$};
        \draw[thick] (.2,1)--(-.2,1) node[left] {$1$};
        \draw[thick] (.2,2)--(-.2,2) node[left] {$2$};
        \draw[thick] (.2,3)--(-.2,3) node[left] {$3$};
        \filldraw[black] (0,1) circle (3pt) node[above right]{$p_1$};
        \filldraw[black] (0,3) circle (3pt) node[above right]{$p_2$};
    \end{tikzpicture}
    \caption{$a_1=a_2=0$, $b_1\geq 1$, and $b_2\geq b_1+2$.}
    \end{subfigure}%
    \hspace{.03\textwidth}
    \begin{subfigure}{.3\textwidth}
    \centering
    \begin{tikzpicture}[scale = .8]
        \draw[<->, thick] (4,0)--(1,0)--(0,1)--(0,4);
        \draw[dashed] (1,0)--(0,0)--(0,1);
        \draw[thick] (1,.2)--(1,-.2) node[below] {$1$};
        \draw[thick] (2,.2)--(2,-.2) node[below] {$2$};
        \draw[thick] (3,.2)--(3,-.2) node[below] {$3$};
        \draw[thick] (.2,1)--(-.2,1) node[left] {$1$};
        \draw[thick] (.2,2)--(-.2,2) node[left] {$2$};
        \draw[thick] (.2,3)--(-.2,3) node[left] {$3$};
        \filldraw[black] (0,1) circle (3pt) node[above right]{$p_1$};
        \filldraw[black] (1,1) circle (3pt) node[above right]{$p_2$};
    \end{tikzpicture}
    \caption{$a_1=0$, $a_2=1$, $b_1\geq 1$, and $b_2 \geq 1$.}
    \end{subfigure}
    \hspace{.03\textwidth}
    \begin{subfigure}{.3\textwidth}
    \centering
    \begin{tikzpicture}[scale = .8]
        \draw[<->, thick] (4,0)--(1,0)--(0,1)--(0,4);
        \draw[dashed] (1,0)--(0,0)--(0,1);
        \draw[thick] (1,.2)--(1,-.2) node[below] {$1$};
        \draw[thick] (2,.2)--(2,-.2) node[below] {$2$};
        \draw[thick] (3,.2)--(3,-.2) node[below] {$3$};
        \draw[thick] (.2,1)--(-.2,1) node[left] {$1$};
        \draw[thick] (.2,2)--(-.2,2) node[left] {$2$};
        \draw[thick] (.2,3)--(-.2,3) node[left] {$3$};
        \filldraw[black] (0,1) circle (3pt) node[above right]{$p_1$};
        \filldraw[black] (2,0) circle (3pt) node[above right]{$p_2$};
    \end{tikzpicture}
    \caption{$a_1=0$, $a_2\geq 2$, $b_1\geq 1$, and $b_2\geq 0$}
    \end{subfigure}
    \caption{The vertices of the cones appearing appearing as subsets of $\mathcal{P}(2,1)$ in the case $a_1=0$, where the four dimensional vector $(a_1,a_2,b_1,b_2)$ is represented as a pair of points $p_1 = (a_1,b_1)$ and $p_2 =(a_2,b_2)$.}
    \label{fig:a1is0}
    \end{figure}

    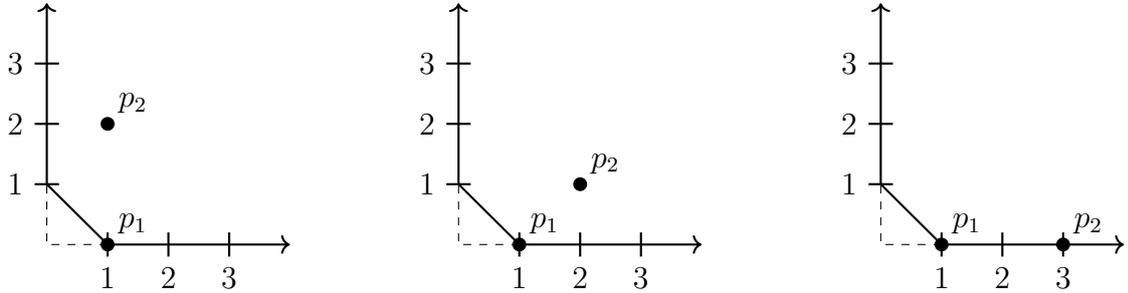
\begin{figure}[h]
    \begin{subfigure}{.3\textwidth}
    \centering
    \begin{tikzpicture}[scale = .8]
        \draw[<->, thick] (4,0)--(1,0)--(0,1)--(0,4);
        \draw[dashed] (1,0)--(0,0)--(0,1);
        \draw[thick] (1,.2)--(1,-.2) node[below] {$1$};
        \draw[thick] (2,.2)--(2,-.2) node[below] {$2$};
        \draw[thick] (3,.2)--(3,-.2) node[below] {$3$};
        \draw[thick] (.2,1)--(-.2,1) node[left] {$1$};
        \draw[thick] (.2,2)--(-.2,2) node[left] {$2$};
        \draw[thick] (.2,3)--(-.2,3) node[left] {$3$};
        \filldraw[black] (1,0) circle (3pt) node[above right]{$p_1$};
        \filldraw[black] (1,2) circle (3pt) node[above right]{$p_2$};
    \end{tikzpicture}
    \caption{$a_1\geq 1$, $a_2=a_1$,  $b_1\geq 0$, and $b_2\geq b_1+2$.}
    \end{subfigure}%
    \hspace{.03\textwidth}
    \begin{subfigure}{.3\textwidth}
    \centering
    \begin{tikzpicture}[scale = .8]
        \draw[<->, thick] (4,0)--(1,0)--(0,1)--(0,4);
        \draw[dashed] (1,0)--(0,0)--(0,1);
        \draw[thick] (1,.2)--(1,-.2) node[below] {$1$};
        \draw[thick] (2,.2)--(2,-.2) node[below] {$2$};
        \draw[thick] (3,.2)--(3,-.2) node[below] {$3$};
        \draw[thick] (.2,1)--(-.2,1) node[left] {$1$};
        \draw[thick] (.2,2)--(-.2,2) node[left] {$2$};
        \draw[thick] (.2,3)--(-.2,3) node[left] {$3$};
        \filldraw[black] (1,0) circle (3pt) node[above right]{$p_1$};
        \filldraw[black] (2,1) circle (3pt) node[above right]{$p_2$};
    \end{tikzpicture}
    \caption{$a_1\geq 1$, $a_2 = a_1+1$, $b_1\geq 0$, and $b_2\geq 1$.}
    \end{subfigure}
    \hspace{.03\textwidth}
    \begin{subfigure}{.3\textwidth}
    \centering
    \begin{tikzpicture}[scale = .8]
        \draw[<->, thick] (4,0)--(1,0)--(0,1)--(0,4);
        \draw[dashed] (1,0)--(0,0)--(0,1);
        \draw[thick] (1,.2)--(1,-.2) node[below] {$1$};
        \draw[thick] (2,.2)--(2,-.2) node[below] {$2$};
        \draw[thick] (3,.2)--(3,-.2) node[below] {$3$};
        \draw[thick] (.2,1)--(-.2,1) node[left] {$1$};
        \draw[thick] (.2,2)--(-.2,2) node[left] {$2$};
        \draw[thick] (.2,3)--(-.2,3) node[left] {$3$};
        \filldraw[black] (1,0) circle (3pt) node[above right]{$p_1$};
        \filldraw[black] (3,0) circle (3pt) node[above right]{$p_2$};
    \end{tikzpicture}
    \caption{$a_1\geq 1$, $a_2\geq a_1+2$, $b_1\geq 0$, and $b_2\geq 0$.}
    \end{subfigure}
    \caption{The vertices of the cones appearing appearing as subsets of $\mathcal{P}(2,1)$ in the case $a_1\geq 1$, where the four dimensional vector $(a_1,a_2,b_1,b_2)$ is represented as a pair of points $p_1 = (a_1,b_1)$ and $p_2 =(a_2,b_2)$.}
    \label{fig:a1greaterthan0}
    \end{figure}
\end{example}


\subsection{Newton-Okounkov Bodies}\label{sec: NO bodies}

Newton-Okounkov bodies are asymptotic invariants of divisors/line bundles usually considered only on projective varieties. They were originally constructed by Okounkov \cite{O}, and their theory was later developed by Kaveh and Khovanskii \cite{KK} and by Lazarsfeld and Musta\c{t}\u{a} \cite{LM}. In this section we show that the construction works for the non-projective variety $\Hilb^{n,n+1}(\C^2)$ and characterize the resulting sets.

We first recall the basic construction \cite{KK,LM}. For a variety $X$ with line bundle $\mathcal{L}$, we have a graded ring of sections
\[ \bigoplus_{d\geq 0}H^0(X,\mathcal{L}^{\otimes d}). \]
The Newton-Okounkov body depends on valuation-like function $\nu:H^0(X,\mathcal{L}^{\otimes d})\setminus\{0\} \to \Z^{\dim(X)}$, satisfying the following properties:
\begin{enumerate}
    \item For any nonzero constant $c$ and nonzero section $s\in H^0(X,\mathcal{L}^{\otimes d})$ we have $\nu(cs)=\nu(s).$
    \item Ordering $\Z^{\dim(X)}$ lexicographically, we have $\nu(s_1+s_2)\geq \min\{ \nu(s_1),\nu(s_2)\}$ for any two nonzero sections $s_1,s_2\in H^0(X,\mathcal{L}^{\otimes d}).$
    \item For nonzero sections $s_1 \in H^0(X,\mathcal{L}^{\otimes d_1})$ and $s_2 \in H^0(X,\mathcal{L}^{\otimes d_2})$ we have $\nu(s_1\otimes s_2) = \nu(s_1)+\nu(s_2)$.
    \item The valuation $\nu$ has one-dimensional leaves. In other words, for every $\mathbf{v}\in \Z^{\dim(x)}$ the quotient of the vector space $\{s\in H^0(X,\mathcal{L}^{\otimes d})|\nu(s)\geq \mathbf{v}\})\cup \{0\}$ by $\{s\in H^0(X,\mathcal{L}^{\otimes d})|\nu(s)> \mathbf{v}\})\cup \{0\}$ is at most one-dimensional.
\end{enumerate}
\begin{definition}
    The Newton-Okounkov body of $\mathcal{L}$ with respect to $\nu$, denoted $\Delta(\mathcal{L})$ or $\Delta_\nu(\mathcal{L})$ is the closed convex hull in $\R^{\dim(X)}$ of the set
    \[ \bigcup_{d\geq 1} \frac{1}{d}\cdot\text{Im}(\nu:H^0(X,\mathcal{L}^{\otimes d})\setminus\{0\}\to \Z^{\dim(X)}). \]
\end{definition}

We will take $X=\Hilb^{n,n+1}(\C^2)$ and $\mathcal{L}=\mathcal{O}(m,k)$. In this case we have $\mathcal{L}^{\otimes d}= \mathcal{O}(dm,dk)$, and by Theorem \ref{thm: sections nested} we know the sections $H^0(\Hilb^{n,n+1}(\C^2),\CO(dm,dk))\simeq A^{dm+dk}[x,y] \cap I^{dk}$. Let $\nu$ be the valuation that takes a nonzero polynomial $f\in A^{dm+dk}[x,y] \cap I^{dk}$ to its lexicographic trailing term exponent $(a,b,a_1,\dots,a_n,b_1,\dots,b_n)$ with variables ordered as in the previous section.  One can check that it satisfies the properties (1)-(4) above.

\begin{definition}
    Let $\Delta(m,k)\subseteq \R^{2n}$ be the set of points $(a_1,\dots,a_n,b_1,\dots,b_n)$ such that:
    \begin{enumerate}
        \item $0\leq a_1\leq a_2 \leq \cdots \leq a_n$, and
        \item for each $j=1,\dots,n,$ we have $b_j\geq \max\{k-a_j,0\} + \sum_{i=1}^{j-1}\max\{m+k-(a_j-a_i),0\}.$
    \end{enumerate}
\end{definition}

Note that $\Delta(m,k)\subseteq \R^{2n}$ is a polyhedron since condition (2) above can be broken up into separate linear inequalities among the coordinates.

\begin{lemma}\label{lem:convexHull}
    \[ \overline{\bigcup_{d\geq 1} \frac{1}{d} \cdot \mathcal{P}(d m, d k)} = \Delta(m,k). \]
\end{lemma}

\begin{proof}
    We first observe that $\mathcal{P}(m,k)\subseteq \Delta(m,k)$ is a subset that contains every integer point in the strict interior of $\Delta(m,k)$. Indeed, for any interior point we have $a_1<a_2<\cdots<a_n$, so condition (2) in the definition of $\mathcal{P}(m,k)$ never applies, and the remaining conditions in the definition of $\mathcal{P}(m,k)$ are exactly the defining equations of $\Delta(m,k)$. By the same argument, we have the same conclusion for $\mathcal{P}(dm,dk)\subseteq \Delta(dm,dk)$ for any $d\geq 1$. Scaling both sets by a factor of $\frac{1}{d}$, we have that $\frac{1}{d}\cdot \mathcal{P}(dm,dk)\subseteq \frac{1}{d}\cdot\Delta(dm,dk)$ is a subset containing every interior point with $\frac{1}{d}\Z$-coordinates of $\frac{1}{d}\cdot\Delta(dm,dk)$.
    
    By the homogeneity of the defining inequalities we have $d\cdot \Delta(m,k) = \Delta(dm,dk)$ for any $d\geq 1$, and hence $\frac{1}{d}\cdot\Delta(dm,dk) = \Delta(m,k)$. Thus, we have shown that 
    \[ \bigcup_{d\geq 1} \frac{1}{d} \cdot \mathcal{P}(d m, d k) \subseteq \Delta(m,k) \]
    is a subset containing every rational point in the strict interior of $\Delta(m,k)$. One easily checks that $\Delta(m,k)\subseteq \R^{2n}$ is a full-dimensional polyhedron (for example, it contains every point with $0\ll a_1\ll a_2\ll \cdots\ll a_n$ and $b_i\geq 0$), so closure of this subset is exactly $\Delta(m,k)$ as desired.
\end{proof}

\begin{corollary}
    The Newton-Okounkov body of $\CO(m,k)$ on $\Hilb^{n,n+1}\C^2$ with respect to $\nu$ is the polyhedron $\R^2_{\geq 0}\times \Delta(m+k,k)\subseteq \R^{2n+2}$.
\end{corollary}

\begin{proof}
    By Corollary \ref{cor:trailingterms}, we have
    \[ \text{Im}(\nu:H^0(\Hilb^{n,n+1}(\C^2),\mathcal{O}(dm,dk))\setminus\{0\}\to \Z^{2n+2}) = \Z^2_{\ge 0}\times \mathcal{P}(dm+dk,dk). \]
    With this characterization, the Newton-Okounkov body is the closed convex hull of the set
    \[ \bigcup_{d \geq 0} \frac{1}{d}\cdot \left(\Z^2_{\ge 0}\times \mathcal{P}(dm+dk,dk)\right). \]
    By Lemma \ref{lem:convexHull}, this is exactly $\R^2_{\ge 0}\times \Delta(m+k,k)$, as desired.
\end{proof}



\end{document}